\documentclass[a4paper,11pt]{article}
\usepackage{braket,hyperref}
\usepackage{amssymb,mathtools,amsthm,amsmath}
\usepackage{tikz}

\usepackage{xargs}
\usepackage{authblk}
\usepackage{enumerate}
\usepackage{enumitem}
\usepackage[colorinlistoftodos,prependcaption,textsize=small]{todonotes}
\usepackage{verbatim}

\theoremstyle{plain}
\newtheorem{theorem}{Theorem}[section]
\newtheorem{claim}[theorem]{Claim}
\newtheorem{conjecture}[theorem]{Conjecture}
\newtheorem{proposition}[theorem]{Proposition}
\newtheorem{corollary}[theorem]{Corollary}
\newtheorem{lemma}[theorem]{Lemma}
\theoremstyle{definition}

\newtheorem{problem}[theorem]{Problem}

\newenvironment{remark}[1][Remark.]{\begin{trivlist}
		\item[\hskip \labelsep {\bfseries #1}]}{\end{trivlist}}

\newcommand{\Rea}{{\mathbb R}}

\DeclareMathOperator{\lk}{lk}

\begin{document}
	\title{Complexes of graphs with bounded independence number}
    \author{Minki Kim\thanks{\href{mailto:kimminki@campus.technion.ac.il}{kimminki@campus.technion.ac.il}. M.K. was supported by ISF grant no. 2023464 and BSF grant no. 2006099.} }
    
    \author{Alan Lew\thanks{\href{mailto:alan@campus.technion.ac.il}{alan@campus.technion.ac.il}. A.L. was supported by ISF grant no. 326/16.}}
	\affil{Department of Mathematics, Technion, Haifa 32000, Israel}

	\date{}
	\maketitle

\begin{abstract}
Let $G=(V,E)$ be a graph and $n$ a positive integer. Let $I_n(G)$ be the abstract simplicial complex whose simplices are the subsets of $V$ that do not contain an independent set of size $n$ in $G$. We study the collapsibility numbers of the complexes $I_n(G)$ for various classes of graphs, focusing on the class of graphs with maximum degree bounded by $\Delta$. As an application, we obtain the following result:

Let $G$ be a claw-free graph with maximum degree at most $\Delta$. Then, every collection of $\left\lfloor\left(\frac{\Delta}{2}+1\right)(n-1)\right\rfloor+1$ independent sets in $G$ has a rainbow independent set of size $n$.
	
\end{abstract}

\section{Introduction}

Let $X$ be a simplicial complex and $d$ a non-negative integer.
A face $\sigma$ that is contained in a unique maximal face $\tau$ of $X$ is called a \emph{free face} of $X$.
 If $\sigma$ is a free face of $X$ with $|\sigma| \le d$, we say that the complex
 \[
 X'= X\setminus \{ \eta\in X:\, \sigma\subset\eta\subset\tau\}
 \]
 is obtained from $X$ by an \emph{elementary $d$-collapse}, and we write
 $
 X\xrightarrow{\sigma} X'.
 $
 A complex $X$ is called $d$-\emph{collapsible} if there exists a sequence of elementary $d$-collapses
 \[
 X=X_1\xrightarrow{\sigma_1} X_2 \xrightarrow{\sigma_2} \cdots \xrightarrow{\sigma_{k-1}} X_k=\emptyset
 \]
 reducing $X$ to the void complex $\emptyset$.

We define the {\em collapsibility number }   $C(X)$ as the minimum integer $d$ such that $X$ is $d$-collapsible.

Let $G=(V,E)$ be a (simple, undirected) graph.
A vertex subset $I \subset V$ is called an {\em independent set} in $G$ if no two vertices in $I$ are adjacent in $G$. 
The {\em independence number} of $G$, denoted by $\alpha(G)$, is the maximal size of an independent set in $G$. For $U\subset V$, we denote by $G[U]$ the subgraph of $G$ induced by $U$.
For every integer $n\geq 1$, we define the simplicial complex
\[
    I_{n}(G)=\{ U\subset V: \, \alpha(G[U])<n\}.
\]
For example, $I_2(G)$ is the clique complex of $G$, i.e. $U \in I_2(G)$ if and only if $G[U]$ is a complete graph. For any graph $G$, the complex $I_1(G)$ is just the empty complex $\{\emptyset\}$.

In this paper we study the collapsibilty numbers of the complexes $I_n(G)$, for several classes of graphs.
Our main motivation is the following problem, presented by Aharoni, Briggs, Kim and Kim in \cite{ABKK}:

Let $\mathcal{F} = \{A_1,\ldots,A_m\}$ be a family of (not necessarily distinct) non-empty subsets of some finite set $V$. For a positive integer $n \leq m$, a {\em rainbow set} of size $n$ for $\mathcal{F}$ is a set of $n$ distinct elements in $V$ of the form $\{a_{i_1},\ldots,a_{i_n}\}$, where $1 \leq i_1 < i_2 < \cdots < i_n \leq m$ and $a_{i_j} \in A_{i_j}$ for each $j \leq n$.

Let $G$ be a graph, and let $\mathcal{F}$ be a finite family of independent sets in $G$. A {\em rainbow independent set} in $G$ with respect to $\mathcal{F}$ is a rainbow set for $\mathcal{F}$ that forms an independent set in $G$. For a positive integer $n$, let $f_G(n)$ be the minimum integer $t$ such that every collection of $t$ independent sets of size $n$ in $G$ has a rainbow independent set of size $n$.
For a graph class $\mathcal{G}$ and a positive integer $n$, let 
\[
f_\mathcal{G}(n)= \sup_{G\in \mathcal{G}} f_G(n).
\]

The connection between the complexes $I_n(G)$ and the parameters $f_G(n)$ is given by the following version of Kalai and Meshulam's ``topological colorful Helly theorem'':

\begin{theorem}[Kalai and Meshulam \cite{KM2005}]\label{km}
	Let $X$ be a $d$-collapsible simplicial complex on vertex set $V$, and let
	$
	X^c=\{\sigma\subset V:\, \sigma\notin X\}.
	$	
	Then, every collection of $d+1$ sets in $X^c$ has a rainbow set belonging to $X^c$.
\end{theorem}

Theorem \ref{km} is a special case of Theorem 2.1 in \cite{KM2005}. A detailed derivation of Theorem \ref{km} from the general case can be found in \cite{aharoni2018fractional}.
An immediate application of Theorem \ref{km} gives us:

\begin{proposition}\label{prop:helly_corollary}
	Let $G$ be a graph and $n\geq 1$ an integer. Then,
	\[
	f_G(n)\leq C(I_n(G))+1.
	\]
\end{proposition}

The study of rainbow independent sets originated as a generalization of the ``rainbow matching problem'' in graphs (note that a matching in a graph is an independent set in its line graph); see e.g. \cite{aharoni2009rainbow,aharoni2016large,barat2017rainbow}. The application of collapsibility numbers in the study of rainbow matchings was initiated in \cite{aharoni2018fractional}, and further developed in \cite{briggs2019choice}.

In \cite{ABKK}, Aharoni et al. prove some results about $f_\mathcal{G}(n)$ for different classes of graphs. One of the main conjectures in \cite{ABKK} is the following.
\begin{conjecture}[Aharoni, Briggs, Kim, Kim \cite{ABKK}]
	\label{mainconj}
	Let $\mathcal{D}(\Delta)$ be the class of graphs with maximum degree at most $\Delta$, and let $n$ be a positive integer. Then,
	\[
	f_{\mathcal{D}(\Delta)}(n) = \left\lceil\frac{\Delta + 1}{2}\right\rceil(n-1) + 1.
	\]
\end{conjecture}
It is shown in \cite{ABKK} that Conjecture~\ref{mainconj} is true for $\Delta\leq 2$ and for $n \leq 3$.
In the general case, the best bounds observed by Aharoni et al. are given by
\[\left\lceil\frac{\Delta + 1}{2}\right\rceil(n-1) + 1 \leq f_{\mathcal{D}(\Delta)}(n) \leq \Delta(n-1)+1.\]

It is natural to ask whether the following extension of Conjecture \ref{mainconj} holds:

\begin{conjecture}\label{mainconj_collapsibility_version}
	Let $G$ be a graph with maximum degree at most $\Delta$, and let $n$ be a positive integer. Then,
		\[
		C(I_n(G))\leq \left\lceil\frac{\Delta + 1}{2}\right\rceil(n-1).
		\]
\end{conjecture}

Our main results are the following:

\begin{theorem}\label{chordal collapse}
	Let $G=(V,E)$ be a chordal graph and $n\geq 1$  an integer. Then,
	\[C(I_n(G)) \leq n-1.\]
\end{theorem}

\begin{theorem}\label{cor:bounded degree}
	Let $G=(V,E)$ be a graph with  maximum degree at most $\Delta$ and  $n\geq 1$ an integer. Then,
	\[
		C(I_n(G))\leq \Delta(n-1).
	\]
\end{theorem}
The bound in Theorem \ref{cor:bounded degree} is tight only for $\Delta\leq 2$. 
In the case $n\leq 3$ we can prove the following tight bounds, for general $\Delta$: 
\begin{theorem}\label{thm:n=2}
	Let $G=(V,E)$ be a graph with maximum degree at most $\Delta$. Then,
	\[
	 C(I_2(G))\leq \left\lceil \frac{\Delta+1}{2} \right\rceil.
	\]
\end{theorem}
\begin{theorem}\label{thm:n=3}
	Let $G=(V,E)$ be a graph with maximum degree at most $\Delta$. Then,
	\[
		C(I_3(G))\leq\begin{cases}
		\Delta+2 & \text{ if $\Delta$ is even,}\\
		\Delta+1 & \text{ if $\Delta$ is odd.}
		\end{cases}
	\]
\end{theorem}

Theorems \ref{cor:bounded degree}, 
\ref{thm:n=2} and \ref{thm:n=3} confirm Conjecture \ref{mainconj_collapsibility_version} in the special cases where $\Delta\leq 2$ or $n\leq 3$. Unfortunately, Conjecture \ref{mainconj_collapsibility_version} does not hold in general: In Section~\ref{sec:lowerbound} we present a family of counterexamples to the case $\Delta=3$.

Combining these results with Proposition \ref{prop:helly_corollary}, we obtain corresponding upper bounds for $f_G(n)$, thus recovering several results first proved in \cite{ABKK}. The following bound, however, is new:

\begin{theorem}\label{maincor}
	Let $G$ be a claw-free graph with maximum degree at most $\Delta$, and let $n\geq 1$ be an integer. Then,
\[
f_{G}(n) \leq \left\lfloor \left(\frac{\Delta}{2}+1\right)(n-1)\right\rfloor+1.
\]
\end{theorem}
Theorem \ref{maincor} shows that Conjecture \ref{mainconj} holds for the subclass of claw-free graphs with maximum degree at most $\Delta$, in the case where $\Delta$ is even. 
The proof of Theorem \ref{maincor} relies on bounding the collapsibility numbers of certain subcomplexes of the complex $I_n(G)$. 

The paper is organized as follows. In Section \ref{sec:prelim} we introduce some basic definitions about graphs and simplicial complexes that we will use throughout the paper. In Section \ref{sec:tools} we present several tools for bounding the collapsibility numbers of a general simplicial complex. Section \ref{sec:chordal} contains the proof of Theorem \ref{chordal collapse}, dealing with the case of chordal graphs. Section \ref{sec:bounded_degree} focuses on the class of graphs with bounded maximum degree. It contains the proofs of Theorems \ref{cor:bounded degree}, \ref{thm:n=2} and \ref{thm:n=3}. In Section \ref{sec:rainbow} we prove our main application, Theorem \ref{maincor}. Section \ref{sec:lowerbound} deals with the \emph{Leray number}, a homological variant of the collapsibility number, of the complex $I_n(G)$. In particular, it
presents extremal examples determining the tightness of our main results (Theorems \ref{thm:n=2}, \ref{thm:n=3} and \ref{maincor}), and 
  examples of $3$-regular graphs for which the complexes $I_n(G)$ do not satisfy the bound in Conjecture \ref{mainconj_collapsibility_version} (for various values of $n$). In Section \ref{sec:open} we discuss some open problems arising from our work and possible directions for further research.

\section{Preliminaries}\label{sec:prelim}

\subsection{Simplicial complexes}
A (finite) \emph{abstract simplicial complex} is a family $X$ of subsets of some finite set that is closed downward. That is, if $\tau\in X$ and $\sigma\subset \tau$, then $\sigma\in X$.

The \emph{vertex set} of $X$ is the set $V=\bigcup_{\sigma\in X} \sigma$. A set $\sigma\in X$ is called a {\em simplex} or a {\em face} of $X$. The \emph{dimension} of a simplex $\sigma\in X$ is $\dim(\sigma)=|\sigma|-1$. The \emph{dimension} of the complex $X$ is the maximal dimension of a simplex in $X$.

A \emph{missing face} of a complex $X$ is a set $\tau\subset V$ such that $\tau\notin X$ but $\sigma\in X$ for any $\sigma\subsetneq \tau$. If all the missing faces of $X$ are of size $2$, then $X$ is called a \emph{flag complex}.

Let $U\subset V$. The subcomplex of $X$ \emph{induced} by $U$ is the complex
\[
	X[U]=\{\sigma\in X:\, \sigma\subset U\}.
\]
For any vertex $v\in V$, we define the \emph{deletion} of $v$ in $X$ to be the subcomplex
\[
X\setminus v = \{\sigma\in X:\, v\notin \sigma\}=X[V\setminus\{v\}].
\]
Let $\tau\subset V$. We define the \emph{link} of $\tau$ in $X$ to be the subcomplex
\[
\lk(X,\tau)= \{\sigma\in X: \, \sigma\cap\tau=\emptyset,\, \sigma\cup\tau \in X\}.
\] 
Note that $\lk(X,\tau)=\emptyset$ unless $\tau\in X$. If $\tau=\{v\}$, we write  $\lk(X,v)=\lk(X,\{v\})$.

Let $X$ and $Y$ be two simplicial complexes on disjoint vertex sets. We define the \emph{join} of $X$ and $Y$ to be the simplicial complex
\[
	X\ast Y= \{ \sigma\cup \tau:\, \sigma\in X, \, \tau\in Y\}.
\]

Let $v\in V$. If $v\in \tau$ for every maximal face $\tau\in X$ we say that $X$ is a \emph{cone over $v$}.

For $U\subset V$, we denote by $2^U=\{\sigma:\, \sigma\subset U\}$ the \emph{complete complex} on vertex set $U$.

\subsection{Homology and Leray numbers}

For $i\geq -1$, let $\tilde{H}_i(X)$ be the $i$-th reduced homology group of $X$ with real coefficients. We say that $X$ is $d$-\emph{Leray} if every induced subcomplex $Y$ of $X$ has trivial homology in dimensions $d$ and above, namely $\tilde{H}_i(Y)=0$ for $i \ge d$. The \emph{Leray number} of $X$, denoted by $L(X)$, is the minimum integer $d$ such that $X$ is $d$-Leray.

The notions of $d$-collapsibility and $d$-Lerayness of simplicial complexes were introduced by Wegner in \cite{Weg75}. He observed the following simple fact:
\begin{lemma}[Wegner \cite{Weg75}]\label{lem:coll_leray}
	Let $X$ be a simplicial complex. Then,
	\[
		C(X)\geq L(X).
	\]
\end{lemma}

In Section \ref{sec:lowerbound} we will use some well known facts about the homology of simplicial complexes (see e.g. \cite{bjorner1995topological}):

\begin{theorem}\label{thm:kunneth}
	Let $X=X_1\ast X_2\ast \cdots \ast X_m$. Then,
	\[
	\tilde{H}_i(X)\cong \bigoplus_{\substack{i_1 +\cdots +i_m =i-m+1,\\
		-1\leq i_j \leq \dim(X_j)\,\, \forall j\in [m]}} \tilde{H}_{i_1}(X_1)\otimes  \cdots \otimes \tilde{H}_{i_m}(X_m).
	\]	
\end{theorem}

Let $\mathcal{F}=\{A_1,\ldots, A_m\}$ be a family of sets. The \emph{nerve} of $\mathcal{F}$ is the simplicial complex
\[
	N(\mathcal{F})=\{ I\subset [m] :\, \cap_{i\in I} A_i\neq \emptyset\}.
\]

The following is a simple version of the Nerve Theorem:
\begin{theorem}\label{thm:nerve}
	Let $X$ be a simplicial complex with maximal faces $\sigma_1,\ldots,\sigma_m$. Then,
	\[
		\tilde{H}_i(X)\cong \tilde{H}_i(N(\{\sigma_1,\ldots,\sigma_m\}))
	\]
	for all $i\geq -1$.
\end{theorem}

Let $X$ be a simplicial complex on vertex set $V$. The {\em combinatorial Alexander dual} of $X$ is the complex
\[
D(X) = \{\sigma \subset V: V \setminus \sigma \notin X\}.
\]
It is easy to check that the maximal faces of $D(X)$ are the complements of the missing faces of $X$. Similarly, the missing faces of $D(X)$ are the complements of the maximal faces of $X$. Alexander duality relates the homology groups of $X$ with those of $D(X)$ (see e.g. \cite{bjorner2009note}):
\begin{theorem}[Alexander duality]\label{thm:alexander}
	If $V\notin X$ then for all $-1\leq i\leq |V|-2$
	\[
		\tilde{H}_i(D(X))\cong \tilde{H}_{|V|-i-3}(X).
	\]
\end{theorem}

\subsection{Graphs}

Throughout this paper, we assume every graph is simple and undirected.
Let $G = (V,E)$ be a graph.

For a vertex subset $U \subset V$, the subgraph of $G$ \emph{induced} by $U$ is the graph $G[U] = \left(U, \{e\in E:\, e\subset U\}\right).$
The subset $U \subset V$ is called a \emph{clique} in $G$ if the induced subgraph $G[U]$ is the complete graph on vertex set $U$.

For any vertex $v \in V$, we define the \emph{deletion} of $v$ in $G$ to be the induced subgraph $G \setminus v = G[V\setminus\{v\}]$.

For each $v \in V$, we define the \emph{open neighborhood} of $v$ in $G$ as the vertex subset \[N_G(v) = \{u \in V: u\text{ is adjacent to }v\},\] and we define the \emph{closed neighborhood} of $v$ in $G$ as \[N_G[v] = \{v\} \cup N_G(v).\]
For a set $A\subset V$, let \[N_G(A)=\bigcup_{u\in A} N_G(u).\]
A vertex $v\in V$ is called a \emph{simplicial vertex} if $N_G[v]$ is a clique.
The \emph{degree} of $v$ in $G$ is the number $\deg_G(v) = |N_G(v)|$.

We say $G$ is \emph{$k$-colorable} (or $k$-partite) if we can partition the vertex set $V$ into $k$ parts so that each part is independent in $G$.
The following is a classical result in graph theory that states a relation between the maximum degree and the $k$-colorability of $G$.
\begin{theorem}[Brooks' Theorem~\cite{brooks1941colouring}]\label{thm:brooks}
	Let $G$ be a connected graph with maximum degree $k$.
	Then $G$ is $k$-colorable unless $G$ is the complete graph $K_{k+1}$ or an odd cycle.
\end{theorem}

The complete bipartite graph $K_{1,3}$ is called a \emph{claw}.
A graph is said to be \emph{claw-free} if it does not have a claw as an induced subgraph.

We say a graph is \emph{chordal} if it does not contain a cycle of length at least $4$ as an induced subgraph. Chordal graphs satisfy the following special property:
\begin{theorem}[Lekkerkerker, Boland \cite{lekkeikerker1962representation}]
Every chordal graph contains a simplicial vertex.
\end{theorem}

\section{Upper bounds for collapsibility numbers}\label{sec:tools}

In this section we present our main technical tools for proving $d$-collapsibility of a simplicial complex.
Most of the bounds presented in this section rely on the inductive application of the following two basic results, due to Tancer:

\begin{lemma}[Tancer {\cite[Prop.1.2]{tancer2011strong}}]
\label{lemma:tancer}
Let $X$ be a simplicial complex on vertex set $V$, and let $v\in V$. Then,
\[
    C(X)\leq \max \{ C(X\setminus v), C(\lk(X,v))+1\}.
\]
\end{lemma}

\begin{lemma}[Tancer {\cite[Prop. 3.1]{tancer2011strong}}]
\label{lemma:conepoint}
Let $X$ be a simplicial complex on vertex set $V$, and let $v\in V$ such that $X$ is a cone over $v$. Then,
\[
    C(X)=C(X\setminus v).
\]
\end{lemma}

It will be helpful to state the following straightforward generalization of Lemma \ref{lemma:conepoint}.
\begin{lemma}\label{lemma:join_with_simplex}
Let $U$ and $V$ be two disjoint sets, and let $X$ be a simplicial complex on vertex set $V$. Then,
\[	
	C(X\ast 2^U)=C(X).
\]	
\end{lemma}
\begin{proof}
	We argue by induction on $|U|$. If $|U|=0$ then $X\ast 2^U=X$, so the claim holds. Now, assume $|U|>1$. Let $u\in U$ and $U'=U\setminus \{u\}$. Note that $X \ast 2^{U}$ is a cone over $u$; therefore, by Lemma \ref{lemma:conepoint},
	\[
		C(X\ast 2^U)= C((X\ast 2^U) \setminus u)= C(X\ast 2^{U'}).
	\]
	 By the induction hypothesis, $C(X\ast 2^{U'})=C(X)$. Hence,
$
	C(X\ast 2^U)= C(X),
$
as wanted.
\end{proof}

\begin{lemma}
	\label{lemma:induction_step}
	Let $X$ be a simplicial complex, and let $\sigma=\{v_1,\ldots,v_k\}\in X$.
	For every $0\leq i\leq k$, define $\sigma_i=\{v_j :\, 1\leq j\leq i\}$. Let $d\geq k$.
	If for all $0\leq i\leq k-1$,
	\[
	C(\lk(X\setminus v_{i+1},\sigma_i))\leq d-i,
	\]
	and
	\[
	C(\lk(X,\sigma_k))\leq d-k,
	\]
	then $C(X)\leq d$.
\end{lemma}
\begin{proof}

We will show that, for any $i\in\{0,\ldots,k\}$,
\[
    C(\lk(X,\sigma_i))\leq d-i.
\]
We argue by backwards induction on $i$. For $i=k$, $C(\lk(X,\sigma_k))\leq d-k$ by assumption. Let $i<k$. By Lemma \ref{lemma:tancer}, we have
\[
    C(\lk(X,\sigma_i))\leq \max \{ C(\lk(X\setminus v_{i+1},\sigma_i)), C(\lk(X,\sigma_{i+1}))+1\}.
\]
But $C(\lk(X\setminus v_{i+1},\sigma_i))\leq d-i$ by assumption, and $C(\lk(X,\sigma_{i+1}))\leq d-i-1$ by the induction hypothesis. Therefore,
\[
    C(\lk(X,\sigma_i))\leq d-i.
\]

Setting $i=0$, we obtain (since $\sigma_0=\emptyset$),
\[
    C(X)=C(\lk(X,\sigma_0))\leq d-0=d.
\]
\end{proof}

As a consequence of Lemma~\ref{lemma:induction_step}, we obtain:

\begin{proposition}
\label{lemma:vertexset_bound}
Let $X$ be a simplicial complex on vertex set $V$. If all the missing faces of $X$ are of dimension at most $d$, then
\[
    C(X)\leq \left\lfloor\frac{d  |V|}{d+1}\right\rfloor.
\]
Moreover, equality $C(X)=\frac{d|V|}{d+1}$ is obtained if and only if $X$ is the join of $r=\frac{|V|}{d+1}$ disjoint copies of the boundary of a $d$-dimensional simplex 
(or equivalently, if the set of missing faces of $X$ consists of $r$ disjoint sets of size $d+1$). 
\end{proposition}
\begin{proof}
We argue by induction on $|V|$. If $|V|=0$, then $X$ is $0$-collapsible, and the inequality holds. Assume $|V|>0$. If $X$ is a complete complex, then it is $0$-collapsible, and the inequality holds. Otherwise, let $\sigma=\{v_1,\ldots,v_{k+1}\}\subset V$ be a missing face of $X$. Since all the missing faces of $X$ are of dimension at most $d$, we have $k\leq d$. For each $0\leq i\leq k$, let $\sigma_i=\{v_j:\, 1\leq j\leq i\}\in X$. Let $V_i$ be the vertex set of $\lk(X\setminus v_{i+1},\sigma_i)$.
Note that for every $0\leq i\leq k$, 
\[
V_i\subset V\setminus\sigma_{i+1}.
\]
Therefore, by the induction hypothesis,
\[
C(\lk(X\setminus v_{i+1},\sigma_i))\leq \frac{d}{d+1}|V_i|\leq \frac{d}{d+1}(|V|-i-1).
\]
Since $i\leq k\leq d$, we obtain
\[
C(\lk(X\setminus v_{i+1},\sigma_i))\leq \frac{d}{d+1}|V|-\frac{i}{i+1}(i+1)
=\frac{d}{d+1}|V|-i.
\]
Also, since $\sigma$ is a missing face, we have
\[
    \lk(X,\sigma_k)= \lk(X\setminus v_{k+1}, \sigma_k),
\]
and in particular $C(\lk(X,\sigma_k))\leq \frac{d}{d+1}|V|-k$. Therefore, by Lemma \ref{lemma:induction_step}, we obtain
\[
    C(X)\leq \frac{d}{d+1}|V|.
\]
Since $C(X)$ is an integer, we obtain $C(X)\leq \left\lfloor\frac{d  |V|}{d+1}\right\rfloor$.

Now, assume $C(X)=\frac{d}{d+1}|V|$. Note that, since $C(X)$ is an integer and $\text{gcd}(d,d+1)=1$, then $d+1$ must divide $|V|$.

Then, there exists some $0\leq i\leq k$ such that
\[
C(\lk(X\setminus v_{i+1},\sigma_i))=\frac{d}{d+1}(|V|-i-1)
\]
(otherwise, by the same argument as above, we could prove that $C(X)<\frac{d}{d+1}|V|$).
Since $d+1$ divides $|V|$, it must also divide $i+1$. Hence, we must have $i=k=d$. By the induction hypothesis, the missing faces of 
\[
    \lk(X,\sigma_d)=\lk(X\setminus v_{d+1}, \sigma_d)
\]
form a set of $r-1$ disjoint sets of size $d+1$. Therefore, the set of missing faces of $X$ consists of $r$ disjoint sets of size $d+1$ plus, possibly, some additional faces of the form $\tau\cup\{v_{d+1}\}$, where $\tau\in V\setminus \sigma$. But the choice of the order $v_1,\ldots,v_{d+1}$ on the vertices of $\sigma$ was arbitrary. Thus, repeating the same argument with a different order (e.g. $v_i'=v_i$ for $i\leq d-1$, $v_d'=v_{d+1}$, $v_{d+1}'=v_{d}$), we obtain that the set of missing faces of $X$ consists exactly of $r$ disjoint sets of size $d+1$.
\end{proof}

\begin{remark}
An analogous bound in terms of Leray numbers was proved in \cite[Prop. 5.4]{adamaszek2014extremal}.
\end{remark}

\begin{lemma}\label{lemma:induction_step2}
Let $X$ be a complex on vertex set $V$, and let $B\subset V$. Let $<$ be a linear order on the vertices of $B$.
Let $\mathcal{P}=\mathcal{P}(X,B)$ be the family of partitions $(B_1,B_2)$ of $B$ satisfying:
\begin{itemize}
\item $B_2\in X$.
\item For any $v\in B_2$, the complex
\[
    \lk(X[V\setminus\{u\in B_1:\,  u<v\}],\{u\in B_2:\,  u<v\})
\]
is not a cone over $v$.
\end{itemize}
If for every $(B_1,B_2)\in\mathcal{P}$,
\[
    C(\lk(X[V\setminus B_1],B_2))\leq d-|B_2|,
\]
then $C(X)\leq d$.
\end{lemma}
\begin{proof}
We argue by induction on $|B|$. If $|B|=0$ there is nothing to prove. So, assume $|B|>0$, and let $v$ be the minimal vertex in $B$ (with respect to the order $<$). 
Let $X'=X\setminus v$, and let $V'=V\setminus\{v\}$ be its vertex set.
Let $B'=B\setminus\{v\}$, and let $(B_1',B_2')\in \mathcal{P}(X', B')$. Define $B_1=B_1'\cup\{v\}$ and $B_2=B_2'$. Then, $B_2\in X\setminus v \subset X$, and for any $u\in B_2$, the complex
\begin{multline*}
    \lk(X[V\setminus \{w\in B_1: \, w<u\}],\{w\in B_2: \, w<u\})
    \\
    =\lk(X'[V'\setminus \{w\in B_1': \, w<u\}],\{w\in B_2': \, w<u\})
\end{multline*}
is not a cone over $u$ (since $(B_1',B_2')\in \mathcal{P}(X',B')$). Therefore $(B_1,B_2)\in \mathcal{P}(X,B)$. So,
\[
   C(\lk(X'[V'\setminus B'_1],B'_2))
   =C(\lk(X[V\setminus B_1],B_2))
   \leq d-|B_2|=d-|B_2'|.
\]
Hence, by the induction hypothesis, $C(X\setminus v)=C(X')\leq d$.

If $X$ is a cone over $v$ then, by Lemma \ref{lemma:conepoint}, $C(X)=C(X\setminus v)\leq d$, as wanted. Otherwise, let $X''=\lk(X,v)$, and let $V''\subset V\setminus \{v\}$ be its vertex set. Let $B''=B\cap V''$, and let $(B_1'',B_2'')\in \mathcal{P}(X'',B'')$. Let $B_2=B_2''\cup\{v\}$ and $B_1=B\setminus B_2$.

Since $B_2''\in X''=\lk(X,v)$, we have  $B_2= B_2''\cup\{v\}\in X$. Let $u\in B_2$. If $u=v$, then
\[
    \lk(X[V\setminus\{w\in B_1:\, w<u\}],\{w\in B_2:\, w<u\})= X
\]
is not a cone over $u=v$. If $u>v$, then
\begin{multline*}
    \lk(X[V\setminus \{w\in B_1: \, w<u\}],\{w\in B_2: \, w<u\})
    \\
    =\lk(X''[V''\setminus \{w\in B_1'': \, w<u\}],\{w\in B_2'': \, w<u\})
\end{multline*}
is not a cone over $u$ (since $(B_1'',B_2'')\in \mathcal{P}(X'',B'')$). Therefore, $(B_1,B_2)\in \mathcal{P}(X,B)$. So, 
\[
   C(\lk(X''[V''\setminus B''_1],B''_2))
   =C(\lk(X[V\setminus B_1],B_2))
   \leq d-|B_2|=(d-1)-|B_2''|.
\]
Thus, by the induction hypothesis, $C(\lk(X,v))=C(X'')\leq d-1$. Hence, by Lemma \ref{lemma:tancer}, $C(X)\leq d$.
\end{proof}

The following bound is proved in \cite{khmel}. For completeness, we include here the proof.

\begin{proposition}[Khmelnitsky \cite{khmel}]\label{lem:link_coll}
	Let $X$ be a simplicial complex on vertex set $V$, and let $\sigma\in X$.
	Then,
	\[
		C(\lk(X,\sigma))\leq C(X).
	\]
\end{proposition}
\begin{proof}
Let $d\geq 0$, and assume that $X$ can be reduced to the void complex by a sequence of $k$ elementary $d$-collapses. We will show that $\lk(X,\sigma)$ is $d$-collapsible. We argue by induction on $k$. If $k=1$, we must have $d=0$ and $X=2^V$; so, $C(\lk(X,\sigma))=C(2^{V\setminus \sigma})=0= d$, and the claim holds.
Assume $k>1$. Then, there exists a free face $\eta\in X$ with $|\eta|\leq d$, such that the complex
\[
	X'=X\setminus\{\xi\in X:\, \eta\subset \xi\}
\]
can be reduced to the void complex by a sequence of $k-1$ elementary $d$-collapses.

Let $\tau$ be the unique maximal face of $X$ containing $\eta$.

Assume that $\eta\subset\sigma$. Let $\xi\in \lk(X,\sigma)$. Then $\eta\subset\sigma\cup\xi\in X$. Therefore, since $\eta$ is contained in the unique maximal face $\tau\in X$, we have $\sigma\cup \xi\subset \tau$. So, $\xi\subset \tau\setminus \sigma$. Since $\tau\setminus \sigma\in \lk(X,\sigma)$, we obtain $\lk(X,\sigma)=2^{\tau\setminus\sigma}$. In particular, $C(\lk(X,\sigma))=0\leq d$.

Otherwise, assume $\eta\not\subset\sigma$. We divide into two cases:
\begin{enumerate}
	\item If $\eta\notin \lk(X,\sigma)$, then
	\[
		\lk(X',\sigma)=\{\xi\in \lk(X,\sigma):\, \eta\not\subset\xi \}	
	=	\lk(X,\sigma).
	\]
	Thus, by the induction hypothesis, $\lk(X,\sigma)$ is $d$-collapsible.
	
	\item If $\eta\in \lk(X,\sigma)$, let $\tau'=\tau\setminus \sigma$. Let $\eta\subset\xi\in \lk(X,\sigma)$. Then $\eta\subset \sigma\cup\xi \in X$; therefore, since $\eta$ is contained in the unique maximal face $\tau$ of $X$, we obtain $\sigma\cup\xi\subset \tau$. That is, $\xi\subset \tau'$. Hence, since $\tau'\in \lk(X,\sigma)$, $\eta$ is a free face in $\lk(X,\sigma)$. So, we can perform the elementary $d$-collapse
	\[
		\lk(X,\sigma)\xrightarrow{\eta} \lk(X,\sigma)\setminus\{\xi\in \lk(X,\sigma):\, \eta\subset\xi\}= \lk(X',\sigma).
	\]
	By the induction hypothesis, $\lk(X',\sigma)$ is $d$-collapsible. Thus, $\lk(X,\sigma)$ is $d$-collapsible.	
\end{enumerate}
\end{proof}

Lastly, we will need the following simple bound:
\begin{lemma}\label{lemma:bdddim}
	Let $X$ be a simplicial complex. Then,
	\[
		C(X)\leq \dim(X)+1.
	\]
\end{lemma}
\begin{proof}
We argue by induction on the number of faces of $X$. If $X$ contains a unique  face (that is, $X=\{\emptyset\}$), then $C(X)=0= \dim(X)+1$.

Now, assume that $X$ contains more that one face. Let $d=\dim(X)+1$.
 Let $\tau$ be a maximal face of $X$. Then, $\tau$ is a free face in $X$ of size $|\tau|\leq d$; so, we can perform the elementary $d$-collapse
 \[
	 X \xrightarrow{\tau} X'=X\setminus \{\tau\}.
\]
By the induction hypothesis, $X'$ is $d$-collapsible. Hence, $X$ is $d$-collapsible. That is,
\[
	C(X)\leq d= \dim(X)+1.
\]
\end{proof}

\bigskip

\section{Chordal graphs}\label{sec:chordal}

In this section we prove Theorem \ref{chordal collapse}, which bounds the collapsibility of $I_n(G)$ in the case that $G$ is a chordal graph. The proof relies on the next result.

\begin{lemma}\label{lemma simplical}
	Let $G=(V,E)$ be a graph, and let $v\in V$ be a simplicial vertex in $G$. Then, for any $n\geq 2$, 
\[
	C(I_n(G))\leq \max \{ C(I_n(G\setminus v)), C(I_{n-1}(G[V\setminus N_G[v]]))+1\}.
\]
\end{lemma}
\begin{proof}
	Let $W \subset V \setminus \{v\}$. Then, $W$ belongs to $\lk(I_n(G), v)$ if and only if $W \setminus N_G(v) \in I_{n-1}(G)$. Indeed, assume that $W\setminus N_G(v)\notin I_{n-1}(G)$; that is,  $G[W\setminus N_G(v)]$ contains an independent set $A$ of size $n-1$. Then,  $A \cup \{v\}$ is an independent set of size $n$ in $G$, and therefore $W \notin \lk(I_n(G), v)$.
	For the opposite direction, suppose $W\notin \lk(I_n(G),v)$. Then, 
	$W \cup \{v\}$ contains an independent set $A$ of size $n$ in $G$. Since $N_G[v]$ is a clique in $G$, $A$ contains at most one vertex from $N_G[v]$. Thus,  $A \setminus N_G[v]\subset W\setminus N_G(v)$ is an independent set of size at least $n-1$. So, $W\setminus N_G(v)\notin I_{n-1}(G)$.

	It follows that $\lk(I_n(G),v) = 2^{N_G(v)} * I_{n-1}(G[V\setminus N_G[v]])$.
	By  Lemma~\ref{lemma:join_with_simplex}, we have
	\[
		C(\lk(I_n(G),v)) = C(I_{n-1}(G[V\setminus N_G[v]])).
	\]
	So, by Lemma \ref{lemma:tancer}, we obtain
	\begin{align*}
	C(I_n(G))
		&\leq	\max \{ C(I_n(G\setminus v)), C(\lk(I_n(G),v))+1\}
	\\&= \max \{ C(I_n(G\setminus v)), C(I_{n-1}(G[V\setminus N_G[v]]))+1\}.
	\end{align*}
\end{proof}

\begin{proof}[Proof of Theorem~\ref{chordal collapse}]

We argue by induction on $|V|$. For $|V|=0$ the statement is obvious. Suppose $|V| > 0$. For $n=1$, $C(I_1(G))=C(\{\emptyset\})=0$, so the claim holds. Let $n\geq 2$. Since $G$ is a chordal graph, there exists a simplicial vertex $v$ in $G$. 
By the induction hypothesis, \[C(I_n(G-v))\leq n-1\] and \[C(I_{n-1}(G[V\setminus N_G[v]]))\leq n-2.\]
Hence, by Lemma~\ref{lemma simplical}, 
\[
	C(I_n(G))\leq  \max \{ C(I_n(G\setminus v)), C(I_{n-1}(G[V\setminus N_G[v]]))+1\} \leq n-1.
\]
\end{proof}

\begin{remark}
Let $G$ be a graph with $\alpha(G)\geq n$, and let $A$ be an independent set of size $n$ in $G$. Then $I_n(G)[A]$ is the boundary of an $(n-1)$-dimensional simplex, and in particular $\tilde{H}_{n-2}(I_n(G)[A])\neq 0$. 
Hence,  $C(I_n(G))\geq L(I_n(G))\geq n-1$. So, the bound in Theorem \ref{chordal collapse} is tight: any chordal graph $G$ with $\alpha(G) \geq n$ has $C(I_n(G)) = n-1$.
\end{remark}

\section{Graphs with bounded maximum degree}\label{sec:bounded_degree}

In this section we prove our main results about graphs with bounded maximum degree, Theorems \ref{cor:bounded degree}, \ref{thm:n=2} and \ref{thm:n=3}. We also prove an auxiliary result about claw-free graphs (Proposition \ref{prop:clawfree}), which will be later used for the proof of Theorem \ref{maincor}.

We begin with the following related problem: Let $\mathcal{X}(k)$ be the class of all $k$-colorable graphs.
In \cite{ABKK} it was observed that $f_{\mathcal{X}(k)}(n) = k(n-1)+1$. 
The following proposition (combined with Proposition \ref{prop:helly_corollary}) offers an alternative proof for this result.

\begin{proposition}\label{prop:easy}
	Let $G$ be a $k$-colorable graph and $n\geq 1$ an integer. Then,
	\[
		C(I_n(G))\leq k(n-1).
	\]
\end{proposition}
\begin{proof}
	Take a proper vertex-coloring of $G$ with $k$ colors. Note that each color class forms an independent set in $G$. Let $\sigma\in I_n(G)$. Since  $\sigma$ contains no independent set of size $n$ in $G$, it contains at most $n-1$ vertices from each color class. It follows that $|\sigma| \leq k(n-1)$. Hence, by Lemma~\ref{lemma:bdddim},
	\[
		C(I_n(G))\leq \dim(I_n(G))+1 \leq k(n-1).
	\]
\end{proof}

Next, we present the proof of Theorem \ref{cor:bounded degree}. We deal with the case $\Delta = 2$ separately:
\begin{theorem}\label{delta=2}
Let $G=(V,E)$ be a graph with maximum degree at most $2$ and $n\geq 1$ an integer. Then $I_n(G)$ is $2(n-1)$-collapsible.
\end{theorem}

Recall that a graph with maximum degree bounded by $2$ is a disjoint union of cycles and paths. In other to apply an inductive argument, we state the following more general claim:

\begin{proposition}\label{claim:delta2}
Let $G=(V,E)$ be a graph with maximum degree at most $2$.
Let $A$ be an independent set in $G$ of size at most $n-1$ that is contained in the union of all the components of $G$ that are paths.
Then,
\[
    C(\lk(I_n(G), A))\leq 2(n-1)-|A|.
\]
\end{proposition}
\begin{proof}

We argue by induction on the number of cycles $c$ in $G$. 

If $c=0$, then $G$ is a disjoint union of paths. In particular, it is a chordal graph, and by Theorem \ref{chordal collapse}, $C(I_n(G))\leq n-1$.
By Proposition \ref{lem:link_coll}, we obtain
\[
C(\lk(I_n(G),A))\leq C(I_n(G))\leq n-1\leq 2(n-1)-|A|.
\]

Let $c\geq 1$, and assume that the claim holds for all graphs with less than $c$ cycles.
Let $C=\{v_1,\ldots,v_k\}$ be the vertex set of a cycle in $G$ (such that $\{v_i,v_{i+1}\}\in E$ for all $i\in [k]$, where the indices are taken modulo $k$).
Let \[r=\min\left\{\left\lfloor\frac{k}{2}\right\rfloor,n-|A|-1\right\},\] and let
\[
U= \{v_{2i-1} :\, 1\leq i\leq r\}.
\]
So, $U$ is an independent set in $G$ of size $r$.

For each $0\leq i\leq  r$, let $U_i=\{v_{2j-1} :\, 1\leq j\leq i\}$.
Let $0\leq i\leq r-1$. The graph $G\setminus v_{2i+1}$ has $c-1$ cycles, and the set $A\cup U_i$ is an independent set contained in components of $G\setminus v_{2i+1}$ that are paths. Therefore, by the induction hypothesis,
\[
C(\lk(I_n(G\setminus v_{2i+1}),A\cup U_i))\leq 2(n-1)-|A|-i.
\]
Next, we divide into two cases. First, assume $r=n-|A|-1<\left\lfloor\frac{k}{2}\right\rfloor$. Then $2r+1\leq k$, and, by the same argument as before, we obtain
\[
C(\lk(I_n(G\setminus v_{2r+1}),A\cup U_r))\leq 2(n-1)-|A|-r.
\]
Since $r=n-|A|-1$, the set $A\cup U_r\cup \{v_{2r+1}\}$ is an independent set of size $n$ in $G$; therefore, $v_{2r+1}\notin \lk(I_n(G),A\cup U_r)$. Hence,
\[
\lk(I_n(G),A\cup U_r)=\lk(I_n(G\setminus v_{2r+1}),A\cup U_r).
\]
So,
\[
C(\lk(I_n(G),A\cup U_r))\leq 2(n-1)-|A|-r.
\]
Now, assume $r=\left\lfloor\frac{k}{2}\right\rfloor$. Then, $U_r$ is a maximum independent set in $G[C]$, and we have
\[
    \lk(I_n(G),A\cup U_r)
    = 2^{C\setminus U_r} \ast \lk(I_{n-r}(G[V\setminus C]), A). 
\]
Therefore, by Lemma \ref{lemma:join_with_simplex}, we obtain
\begin{align*}
   C(\lk(I_n(G),A\cup U_r))
    &= C(\lk(I_{n-r}(G[V\setminus C]),A)) \leq 2(n-r-1)-|A|
    \\&= 2(n-1)-|A|-2r \leq 2(n-1)-|A|-r,
\end{align*}
where the first inequality follows by the induction hypothesis (since the number of cycles in $G[V\setminus C]$ is $c-1$).

In both cases we obtained
\[
C(\lk(I_n(G),A\cup U_r))\leq 2(n-1)-|A|-r.
\]
So, by Lemma \ref{lemma:induction_step}, we obtain
\[
    C(\lk(I_n(G),A))\leq 2(n-1)-|A|,
\]
as wanted.

\end{proof}

Theorem \ref{delta=2} follows from Proposition \ref{claim:delta2} by setting $A=\emptyset$.

Now we can prove the general case of Theorem \ref{cor:bounded degree}:

\begin{proof}[Proof of Theorem \ref{cor:bounded degree}]
	
	We argue by induction on $n$. For $n=1$ the claim is trivial. Assume $n\geq 2$.

    If $\Delta=1$ then the edges of $G$ are pairwise disjoint. In particular, $G$ is a chordal graph; therefore, the claim follows from Theorem \ref{chordal collapse}. 
	If $\Delta=2$, the claim follows from Theorem \ref{delta=2}. Assume $\Delta\geq 3$, and let $G$ be a graph with maximum degree at most $\Delta$. We will show that $C(I_n(G))\leq \Delta(n-1)$. Let $c(G)$ be the 
	number of connected components of $G$ that are isomorphic to the complete graph $K_{\Delta+1}$. We argue by induction on $c(G)$.	
	
	 If $c(G)=0$, then by Brooks' Theorem (Theorem \ref{thm:brooks}) $G$ is $\Delta$-colorable. Then, by Proposition \ref{prop:easy}, $I_n(G)$ is $\Delta(n-1)$-collapsible, as wanted.
	
	Otherwise, assume there exists a component of $G$ that is isomorphic to $K_{\Delta+1}$, and let $v$ be a vertex in that component. Note that $v$ is a simplicial vertex in $G$. Since $c(G\setminus v)=c(G)-1$, we obtain by the induction hypothesis
	\[
		C(I_n(G\setminus v))\leq \Delta(n-1).
	\]
	Also, by the (first) induction hypothesis, we have
	\[
		C(I_{n-1}(G[V\setminus N_G[v]]))\leq \Delta(n-2)\leq \Delta(n-1)-1.
	\]
	So, by Lemma \ref{lemma simplical}, we obtain
	\[
		C(I_n(G)) \leq \max\{	C(I_n(G\setminus v)),	C(I_{n-1}(G[V\setminus N_G[v]]))+1\}\leq  \Delta(n-1).
	\]

\end{proof}

\subsection{The $n\leq 3$ case and claw-free graphs}\label{sec:clawfree}

Next, we prove Theorems \ref{thm:n=2} and \ref{thm:n=3}, which give tight upper bounds on the collapsibility of $I_n(G)$ for graphs $G$ with bounded maximum degree, for $n\leq 3$. We also prove Proposition \ref{prop:clawfree}, bounding the collapsibility of certain subcomplexes of $I_n(G)$, in the case where $G$ is a bounded degree claw-free graph.

\begin{proof}[Proof of Theorem \ref{thm:n=2}]

We argue by induction on $|V|$. For $|V|=0$ the bound holds trivially.
Assume $|V|>0$, and let $v\in V$.
By Lemma \ref{lemma:tancer}, we have
\begin{equation}\label{eq:I2}
	C(I_2(G))\leq \max \{ C(I_2(G\setminus v)), C(\lk(I_2(G),v))+1\}.
\end{equation}
Note that $\lk(I_2(G),v)$ is a flag complex on vertex set $N_G(v)$. 
Thus, by Proposition~\ref{lemma:vertexset_bound}, we have
\[
C(\lk(I_2(G), v)) \leq
\left\lfloor\frac{|N_G(v)|}{2}\right\rfloor\leq
 \left\lfloor\frac{\Delta}{2}\right\rfloor = \left\lceil \frac{\Delta+1}{2} \right\rceil-1.
\]
Also, by the induction hypothesis, \[C(I_2(G\setminus v)) \leq \left\lceil \frac{\Delta+1}{2} \right\rceil.\]
Hence, by \eqref{eq:I2}, we obtain
\[
C(I_2(G))\leq \left\lceil \frac{\Delta+1}{2} \right\rceil.
\]
\end{proof}

\begin{lemma}\label{lemma:isflag}
Let $G=(V,E)$ be a graph and $n\geq 2$ an integer. Let $A$ be an independent set of size $n-1$ in $G$, such that any vertex in $V\setminus A$ is adjacent to at most two vertices in $A$. Let
\[
B=\bigcup_{\{u,v\}\in \binom{A}{2}} N_G(u)\cap N_G(v).
\]
Then, $\lk(I_n(G),A\cup B)$ is a flag complex.
\end{lemma}
\begin{proof}
Let $X=\lk(I_n(G),A\cup B)$, and let $\tau$ be a missing face of $X$. Then, there exists an independent set $I$ of $G$ of size $n$, such that $\tau\subset I\subset \tau\cup A\cup B$. We may choose $I$ such that $|A\cap I|$ is maximal. Each vertex in $A\setminus I$ is adjacent to at least two vertices in $I\setminus A$:
otherwise, assume there exists $a\in A\setminus I$ that is adjacent to at most one vertex in $I\setminus A$.  We divide into two cases:
\begin{itemize}
\item 
If $a$ is not adjacent to any vertex in $I \setminus A$, let $\tau'=\tau\setminus\{u\}$ for any vertex $u\in\tau$. 
\item 
If $a$ is adjacent to a single vertex $u\in I \setminus A$, observe that $u$ should  be contained in $\tau$.
If not, we can take an independent set $I' = I \setminus \{u\} \cup \{a\}$ of size $n$ in $G$ such that $\tau \subset I' \subset \tau \cup A \cup B$.
Since $|A \cap I'| = |A \cap I| + 1$, this contradicts the maximality assumption of $|A \cap I|$.
Hence, $u\in \tau$.
Now, let $\tau'=\tau\setminus\{u\}$.
\end{itemize}
In both cases, $I\setminus\{u\}\cup\{a\}$ is an independent set of size $n$ satisfying $\tau'\subset I\setminus\{u\}\cup\{a\}\subset \tau'\cup A\cup B$.
It follows that $\tau'\notin X$, which is a contradiction to $\tau$ being a missing face.

Let $|\tau|=k$ and $|A\cap I|=t$. Then, $|A\setminus I|=n-t-1$; so, there are at least $2(n-t-1)$ edges between $A$ and $I\setminus A$. 

By assumption, each vertex $v\in I\setminus (A\cup \tau)$ is adjacent to at most $2$ vertices in $A$.
Therefore, since $|I\setminus (A\cup \tau)|=n-t-k$,
there are at least $2(n-t-1)-2(n-t-k)=2k-2$ edges between $A$ and $\tau$. But, since $\tau\subset V\setminus B$, each vertex in $\tau$ is adjacent to at most one vertex in $A$. Therefore, we must have $2k-2\leq k$; that is, $|\tau|=k\leq 2$.
Thus, $X$ is a flag complex.
\end{proof}

\begin{proposition}\label{prop:clawfree}
Let $G=(V,E)$ be a claw-free graph with maximum degree at most $\Delta$, and let $n\geq 1$ be an integer. Let $A$ be an independent set of size $n-1$ in $G$. Then,
\[
    C(\lk(I_n(G),A))\leq \left\lfloor\frac{(n-1)\Delta}{2}\right\rfloor.
\]
\end{proposition}
\begin{proof}
For $n=1$ the claim holds trivially. Assume $n\geq 2$.

Let $v\in V\setminus (A\cup N_G(A))$. Then, $A\cup\{v\}$ is an independent set of size $n$ in $G$; hence, $v\notin \lk(I_n(G),A)$. So, we may assume without loss of generality that $V=N_G(A)\cup A$.
Let
\[
B=\bigcup_{\{u,v\}\in \binom{A}{2}} N_G(u)\cap N_G(v)
\]
and $U=N_G(A)\setminus B$. Since $G$ is claw-free, each vertex is adjacent to at most $2$ vertices in $A$. Hence, we have
\[
|N_G(A)|= \sum_{v\in A} |N_G(v)| - \sum_{\{u,v\}\in \binom{A}{2}} |N_G(u)\cap N_G(v)| = \sum_{v\in A} |N_G(v)| - |B|.
\]
So, since the maximum degree in $G$ is at most $\Delta$, we obtain \[|U|\leq (n-1)\Delta-2|B|.\]

Let $(B_1,B_2)$ be a partition of $B$ such that $B_2\in\lk(I_n(G),A)$. Let $G'=G[V\setminus B_1]$, and let
\[
X=\lk(I_n(G)[V\setminus B_1],A\cup B_2)= \lk(I_n(G'),A\cup B_2).
\]
Note that
\[
B_2=\bigcup_{\{u,v\}\in \binom{A}{2}} N_{G'}(u)\cap N_{G'}(v)
\]
Also, since $G'$ is claw-free and $A$ is independent in $G'$, then every vertex in $V\setminus B_1$ is adjacent to at most $2$ vertices in $A$. Therefore, by Lemma \ref{lemma:isflag}, $X$ is a flag complex. 

The vertex set of $X$ is contained in $U=N_G(A)\setminus B$.
Thus, by Proposition \ref{lemma:vertexset_bound}, we obtain
\[
   C(X)\leq \left\lfloor\frac{|U|}{2}\right\rfloor \leq \left\lfloor
\frac{(n-1)\Delta-2|B|}{2}\right\rfloor\leq \left\lfloor\frac{(n-1)\Delta}{2}\right\rfloor-|B_2|.
\]
Therefore, by Lemma \ref{lemma:induction_step2}, \[
C(\lk(I_n(G),A))\leq \left\lfloor\frac{(n-1)\Delta}{2}\right\rfloor.
\]
\end{proof}

\begin{proposition}\label{prop:n3_odd}
Let $G=(V,E)$ be a graph with maximum degree at most $\Delta$. Let $A=\{a_1,a_2\}$ be an independent set of size $2$ in $G$. Assume that there exists an independent set in $G$ of the form $\{a_1,w,w'\}$, where $w,w'\in N_G(a_2)$, or there exists an independent set of the form $\{a_2,v,v'\}$, where $v,v'\in N_G(a_1)$. Then, 
\[
    C(\lk(I_3(G),A))\leq\begin{cases}
        \Delta & \text{ if $\Delta$ is even,}\\
        \Delta-1 & \text{ if $\Delta$ is odd.}
        \end{cases}
\]
\end{proposition}
\begin{proof}

Let $v\in V\setminus (N_G(A)\cup A)$. Then $A\cup\{v\}$ is an independent set of size $3$ in $G$; hence, $v\notin \lk(I_3(G),A)$. So, we may assume without loss of generality that $V=N_G(A)\cup A$.

Let
\[
B= N_G(a_1)\cap N_G(a_2)
\]
and $U=N_G(A)\setminus B$.
Since the maximum degree of a vertex in $G$ is at most $\Delta$, we have
\[
 |N_G(A)|= |N_G(a_1)|+|N_G(a_2)|-|N_G(a_1)\cap N_G(a_2)|
 \leq 2\Delta-|B|.
\]
So, $|U|\leq 2\Delta-2|B|$.

Write $B=\{u_1,\ldots,u_k\}$.
Let $\mathcal{P}=\mathcal{P}(\lk(I_3(G),A),B)$ be the family of partitions $(B_1,B_2)$ of $B$ satisfying:
\begin{itemize}
\item $B_2\in \lk(I_3(G),A)$.
\item For any $u_i\in B_2$, the complex
\[
    \lk(I_3(G)[V\setminus\{u_j\in B_1:\,  j<i\}],A\cup\{u_j\in B_2:\,  j<i\})
\]
is not a cone over $u_i$.
\end{itemize}

Let $(B_1,B_2)\in\mathcal{P}$. Let $G'=G[V\setminus B_1]$, and let
\[
X=\lk(I_3(G)[V\setminus B_1],A\cup B_2)= \lk(I_3(G'),A\cup B_2).
\]
Note that $B_2=N_{G'}(a_1)\cap N_{G'}(a_2)$. Also, since $A$ is of size $2$, then every vertex in $V\setminus B_1$ is adjacent to at most $2$ vertices in $A$. Therefore, by Lemma \ref{lemma:isflag}, $X$ is a flag complex. 

The vertex set of $X$ is contained in $U=N_G(A)\setminus B$.
So, by Proposition \ref{lemma:vertexset_bound}, we obtain
\begin{equation}
\label{eq:I3_1}
   C(X)\leq \frac{|U|}{2}
\leq 
\frac{2\Delta-2|B|}{2} = \Delta - |B| \leq \Delta-|B_2|.
\end{equation}
Therefore, by Lemma \ref{lemma:induction_step2},
\[
C(\lk(I_3(G),A))\leq \Delta.
\]
Now, assume $\Delta$ is odd.
Again, let $(B_1,B_2)\in\mathcal{P}$, and let \[
X=\lk(I_3(G)[V\setminus B_1],A\cup B_2). \]
If $B_2\neq B$ then, by \eqref{eq:I3_1},
\[
    C(X)\leq \Delta- |B| \leq \Delta-1 - |B_2|.
\]
So, assume $B_2=B$. By the equality case of Proposition \ref{lemma:vertexset_bound}, we have $C(X)\leq \Delta-1-|B|$ unless $X$ contains exactly $2\Delta-2|B|$ vertices, and its set of missing faces consists of $\Delta-|B|=\Delta-k$ pairwise disjoint sets of size $2$. 

Assume for contradiction that this is the case.
Then, $X$ is a simplicial complex on vertex set $U=U_1\cup U_2$, where $U_1=N_G(a_1)\setminus N_G(a_2)$ and $U_2=N_G(a_2)\setminus N_G(a_1)$, and $|U_1|=|U_2|=\Delta-k$.

\begin{claim}\label{claim:odd1}
Let $J$ be an independent set of size $3$ in $G$. Then $J$ is of one of the following forms:
\begin{itemize}
    \item $J=\{a_1,v,w\}$, where $v,w\in U_2$,
    \item $J=\{a_2,v,w\}$, where $v,w\in U_1$, or
    \item $J=\{u_i,v,w\}$ for some $i\in[k]$, where $v,w\in U$.
\end{itemize}
\end{claim}
\begin{proof}
Since $B_2 = B$ and $(B_1,B_2)\in\mathcal{P}$, we have $B\in \lk(I_3(G),A)$.
Thus, any independent set $J$ of size $3$ in $G$ contains at least one vertex from $U$.
Also, since $X$ is a flag complex, at least one vertex in $J$ must belong to $A\cup B$ (otherwise $J$ is a missing face of size $3$ in $X$).

Note that since $U\subset N_G(A)$, each independent set of size $3$ contains at most one of the vertices $a_1$ or $a_2$.

Assume that $a_1\in J$. Then, for all $i\in[k]$, $u_i\notin J$. Otherwise, the unique vertex $v$ in $J\setminus \{a_1,u_i\}$ does not belong to $X$, a contradiction to the assumption that the vertex set of $X$ is the whole set $U$. So, the two vertices in $J\setminus\{a_1\}$ must belong to $U$. And, since all the vertices in $U_1$ are adjacent to $a_1$, they must in fact belong to $U_2$, as wanted.

Similarly, if $a_2\in J$, then the two vertices in $J\setminus\{a_2\}$ must belong to $U_1$.

Now, assume that $a_1,a_2\notin J$. Then, there exists some $i\in[k]$ such that $u_i\in J$. For all $j\in[k]\setminus\{i\}$, $u_j\notin J$, otherwise the unique vertex $v$ in $J\setminus\{u_i,u_j\}$ does not belong to $X$, a contradiction to the assumption that the vertex set of $X$ is the whole set $U$. So, the two vertices in $J\setminus \{u_i\}$ must belong to $U$, as wanted.
\end{proof}

\begin{claim}\label{claim:odd2}
There exist distinct vertices $v_1,\ldots,v_k\in U_1$ and $w_1,\ldots,w_k\in U_2$ such that:
\begin{itemize}
    \item  For all $i\in[k]$, $\{u_i,v_i,w_i\}$ is an independent set in $G$.
    \item For all $1\leq j<i\leq k$, $\{u_j,v_i,w_i\}$ is not independent in $G$.
\end{itemize}
\end{claim}
\begin{proof}
We define the vertices $v_1,\ldots,v_k,w_1,\ldots,w_k$ recursively, as follows.
Let $i\in[k]$, and assume that we already defined $v_1,\ldots,v_{i-1}$ and $w_1,\ldots,w_{i-1}$. Since $(B_1,B_2)=(\emptyset,B)\in \mathcal{P}$, then the complex
\[
    X'=\lk(I_3(G),A\cup \{u_j\in B:\, j<i\})
\]
is not a cone over $u_i$. Therefore, there exists a missing face $\tau$ of $X'$ containing $u_i$. Since $\tau$ is a missing face of $X'$, there exists an independent set $J$ of size $3$ in $G$ containing $\tau$.
By Claim \ref{claim:odd1}, $J$ is of the form $J=\{u_i,v_i,w_i\}$, for some $v_i,w_i\in U$.

Note that actually $J=\tau$. Otherwise, assume without loss of generality that $\tau=\{u_i,v_i\}$. Then $w_i\notin X'$. But then $w_i\notin X$, a contradiction to the assumption that the vertex set of $X$ is the whole set $U$.

If both $v_i$ and $w_i$ belong to $U_1$, or both of them belong to $U_2$, then $\{v_i,w_i\}\notin X'$, a contradiction to $\{u_i,v_i,w_i\}$ being a missing face. So, we may assume that $v_i\in U_1$ and $w_i\in U_2$. Moreover, for all $j<i$,  
$\{u_j,v_i,w_i\}$ is not independent in $G$, otherwise $\{v_i,w_i\}\notin X'$, a contradiction to $\{u_i,v_i,w_i\}$ being a missing face.

The pairs $\{\{v_i,w_i\}\}_{i\in[k]}$ are missing faces of the complex $X$. Hence, they must be pairwise disjoint. Thus, the vertices $v_1,\ldots,v_k,w_1,\ldots,w_k$ are all distinct. 
\end{proof}

\begin{claim}\label{claim:odd3}
There exist some $i_0\in[k]$ and vertices $v_{i_0}'\in U_1\setminus\{v_1,\ldots,v_k\}$, $w_{i_0}'\in U_2\setminus\{w_1,\ldots,w_k\}$ such that $\{u_{i_0},v_{i_0}',w_{i_0}'\}$ is independent in $G$.

\end{claim}
\begin{proof}

Recall that, by assumption, the missing faces of $X$ consist of $\Delta-k$ pairwise disjoint sets of size $2$. In particular, each vertex $v\in U$ belongs to exactly one missing face of $X$.

Assume for contradiction that the only missing faces of $X$
of the form $\{v,w\}$, where $v\in U_1$ and $w\in U_2$, are the pairs $\{v_i,w_i\}$, $i\in[k]$, from Claim \ref{claim:odd2}. 

Then, the $\Delta-2k$ remaining missing faces must be of the form $\{v,w\}$, where $v,w\in U_1$ or $v,w\in U_2$. In particular, the set $U_1\setminus \{v_1,\ldots,v_k\}$ must be of even size (otherwise, there exists a vertex $v\in U_1\setminus \{v_1,\ldots,v_k\}$ that does not belong to any missing face of $X$, a contradiction). But \[|U_1\setminus \{v_1\ldots,v_k\}|=\Delta-2k\] 
is odd, since $\Delta$ is odd.

Therefore, there exists some additional missing face of the form $\{v,w\}$, where $v\in U_1$, $w\in U_2$. That is, there is some $i_0\in [k]$ such that $\{u_{i_0},v,w\}$ is independent in $G$. So, we can choose $v_{i_0}'=v$ and $w_{i_0}'=w$.
\end{proof}

\begin{claim}\label{claim:odd6}
$\Delta\geq 2k+3.$
\end{claim}
\begin{proof}
Assume without loss of generality that there exists an independent set in $G$ of the form $\{a_1,w,w'\}$, where $w,w'\in N_G(a_2)$. Then, the set $\{w,w'\}$ is a missing face in $X$. Since the missing faces of $X$ are all disjoint, the vertices $w_1,\ldots,w_k,w_{i_0}',w,w'\in U_2$ must be all distinct. Therefore,
\[
    \Delta-k = |U_2|\geq k+3.
\]
Hence, $\Delta\geq 2k+3$.
\end{proof}

Let
\[
S=\{j\in[k]\setminus\{{i_0}\}: \, \{v_{i_0},u_j\}\notin E \text{ or } \{w_{i_0},u_j\}\notin E\}.
\]

\begin{claim}\label{claim:odd7}
There exists a set $N_1$ consisting of exactly one vertex from each pair $\{w_j,u_j\}$, for all $j\in S$, such that
\[
    N_G(v_{i_0})=\{a_1\}\cup\{w_{i_0}'\}\cup (U_1\setminus \{v_{i_0}\})\cup \{u_j:\, j\in [k]\setminus (S\cup\{{i_0}\})\} \cup N_1.
\]
In particular, $|N_G(v_{i_0})|=\Delta$.

Similarly, there exists a set $N_2$ consisting of exactly one vertex from each pair $\{v_j,u_j\}$, for all $j\in S$, such that
\[
    N_G(w_{i_0})=\{a_2\}\cup\{v_{i_0}'\}\cup (U_2\setminus \{w_{i_0}\})\cup \{u_j:\, j\in [k]\setminus (S\cup\{{i_0}\})\} \cup N_2.
\]
And, in particular, $|N_G(w_{i_0})|=\Delta$.
\end{claim}
\begin{proof}
We prove the claim for $v_{i_0}$. The proof for $w_{i_0}$ is identical.

First, since $v_{i_0}\in U_1$, then $a_1$ is adjacent to $v_{i_0}$. Also, for every $v_{i_0}\neq v\in U_1$, $v$ is adjacent to $v_{i_0}$, since otherwise the set $\{u_2,v_{i_0},v\}$ is independent in $G$, but then the set $\{v,v_{i_0}\}$ is a missing face of $X$ that intersects the missing face $\{v_{i_0},w_{i_0}\}$, a contradiction to the assumption that the missing faces are pairwise disjoint.

The vertex $w_{i_0}'$ must also be adjacent to $v_{i_0}$, otherwise $\{u_{i_0},v_{i_0},w_{i_0}'\}$ is an independent set in $G$. But then, $\{v_{i_0},w_{i_0}'\}$ is a missing face of $X$ intersecting the missing face $\{v_{i_0},w_{i_0}\}$, again a contradiction.

By the definition of $S$, $v_{i_0}$ is adjacent to $u_j$ for all $j\in[k]\setminus(S\cup\{{i_0}\})$.

Finally, let $j\in S$. If $\{v_{i_0},u_j\}\notin E$ and $\{v_{i_0},w_j\}\notin E$, then $\{v_{i_0},u_j,w_j\}$ is independent in $G$; therefore, $\{v_{i_0},w_j\}$ is a missing face of $X$, a contradiction. So, $v_{i_0}$ is adjacent to either $u_j$ or $w_j$. Let $S'=\{j\in S:\, \{u_j,v_{i_0}\}\in E\}$. Let
\[
    N_1=\{u_j:\, j\in S'\}\cup \{w_j:\, j\in S\setminus S'\}.
\]
Then, $N_1\subset N_G(v_{i_0})$. Let 
\[
N=\{a_1\}\cup\{w_{i_0}'\}\cup (U_1\setminus \{v_{i_0}\})\cup \{u_j:\, j\in [k]\setminus (S\cup\{{i_0}\})\} \cup N_1.
\]
We showed that $N\subset N_G(v_{i_0})$. Note that
\[
|N|=1+1+(\Delta-k-1)+(k-|S|-1)+|S|=\Delta.
\]
Since the maximal degree of a vertex in $G$ is at most $\Delta$, then we must have $N_G(v_{i_0})=N$, as wanted.
\end{proof}

\begin{claim}\label{claim:odd4}
For all $j\in [k]\setminus\{i_0\}$, $u_{i_0}$ is adjacent in $G$ to at least one of the vertices $v_j$ or $w_j$.
\end{claim}
\begin{proof}
Let $j\neq {i_0}$.
Assume for contradiction that $u_{i_0}$ is not adjacent to any of the two vertices $v_j$ and $w_j$. Then $\{u_{i_0},v_j,w_j\}$ is independent in $G$. So, by Claim \ref{claim:odd2}, we must have ${i_0}>j$.
Moreover, either $\{v_{i_0},u_j\}\in E$ or $\{w_{i_0}, u_j\}\in E$ (otherwise $\{u_j, v_{i_0},w_{i_0}\}$  is independent in $G$, a contradiction to Claim \ref{claim:odd2}). Assume without loss of generality that $\{v_{i_0},u_j\}\in E$. The vertex $v_{i_0}$ must be also adjacent to $w_j$, since otherwise the set $\{u_{i_0},v_{i_0},w_j\}$ is independent in $G$. But then $\{v_{i_0},w_j\}$ is a missing face of $X$, a contradiction to the assumption that the missing faces are pairwise disjoint.

But, by Claim \ref{claim:odd7}, the set of neighbors of $v_{i_0}$ in $G$, $N_G(v_{i_0})$, contains at most one of the vertices $u_j$ or $w_j$, a contradiction.

So, $u_{i_0}$ must be adjacent in $G$ to at least one of the vertices $v_j$ or $w_j$.
\end{proof}

\begin{claim}\label{claim:odd5}
There is some vertex \[w\in U\setminus (\{v_j,w_j: \, j\in S\}\cup \{v_{i_0},w_{i_0},v_{i_0}',w_{i_0}'\})\] such that $\{u_{i_0},w\}\notin E$.
\end{claim}
\begin{proof}
Let $U'=U\setminus (\{v_j,w_j: \, j\in S\}\cup \{v_{i_0},w_{i_0},v_{i_0}',w_{i_0}'\})$. The vertex $u_{i_0}$ is adjacent in $G$ to both $a_1$ and $a_2$ (since $u_{i_0}\in B=N_G(a_1)\cap N_G(a_2)$). Also, by Claim \ref{claim:odd4}, it is adjacent to at least $|S|$ vertices from the set $\{v_j, w_j:\, j\in S\}$.
By the definition of $S$, for each $j \in S$, $u_j$ is not adjacent to one of the vertices $v_{i_0}$ or $w_{i_0}$. 
Thus $u_{i_0}$ must be adjacent in $G$ to $u_j$ (otherwise, one of the sets $\{u_j,u_{i_0},v_{i_0}\}$ or $\{u_j,u_{i_0},w_{i_0}\}$ is independent in $G$, in contradiction to Claim \ref{claim:odd1}).

So, $u_{i_0}$ is adjacent to at least $2|S|+2$ vertices outside of $U'$. Since the degree of $u_{i_0}$ is at most $\Delta$, $u_{i_0}$ is adjacent to at most $\Delta-2-2|S|$ vertices in $U'$.

But $|U'|=|U|-2|S|-4= 2\Delta-2k-2|S|-4$. So, $u_{i_0}$ is not adjacent to at least $\Delta-2k-2$ vertices in $U'$. By Claim \ref{claim:odd6}, $\Delta\geq 2k+3$. Therefore, $u_{i_0}$ is not adjacent to at least one vertex $w\in U'$.
\end{proof}

Assume without loss of generality that the vertex $w$ from Claim \ref{claim:odd5} belongs to $U_2$. If $\{v_{i_0},w\}\notin E$, then $\{u_{i_0},v_{i_0},w\}$ is independent in $G$. But then, $\{v_{i_0},w\}$ is a missing face of $X$ intersecting $\{v_{i_0},w_{i_0}\}$, a contradiction to the assumption that all the missing faces are disjoint. So, $w\in N_G(v_{i_0})$. But this is a contradiction to Claim \ref{claim:odd7}.

Therefore, $C(X)\leq (\Delta-1)-|B|$; so, by Lemma \ref{lemma:induction_step2}, $\lk(I_3(G),A)$ is $(\Delta-1)$-collapsible.

\end{proof}

\begin{proposition}\label{prop:n3_linka1}
Let $\Delta\geq 2$. Let $G=(V,E)$ be a graph with maximum degree at most $\Delta$, and let $a_1\in V$. Then,
\[
    C(\lk(I_3(G),a_1))\leq
    \begin{cases}
        \Delta+1 & \text{ if $\Delta$ is even,}\\
        \Delta & \text{ if $\Delta$ is odd.}
        \end{cases}
\]
\end{proposition}
\begin{proof}
Let $d=\Delta+2$ if $\Delta$ is even, and $d=\Delta+1$ if $\Delta$ is odd.
Let $V'$ be the vertex set of $\lk(I_3(G),a_1)$.
We argue by induction on $|V'|$. If $|V'|\leq \Delta$, then by Proposition \ref{lemma:vertexset_bound},
\[
C(\lk(I_3(G),a_1))\leq \frac{2|V'|}{3}\leq \frac{2\Delta}{3}\leq d-1,
\]
as wanted.
Otherwise, let $|V'|>\Delta$. We divide into three different cases:
\begin{itemize}
    \item[Case 1:] There exists an independent set in $G$ of the form $\{u,v,a_2\}$, where $u,v\in N_G(a_1)$ and $a_2\notin N_G(a_1)$. Then, by Proposition \ref{prop:n3_odd}, we have 
    \[
        C(\lk(I_3(G),\{a_1,a_2\}))\leq d-2.
    \]

    \item[Case 2:] There exists a triple $\{u,v,a_2\}\subset V'$ such that $u,v,a_2\notin N_G(a_1)$, $\{u,v\}\notin E$, $\{u,a_2\}\in E$ and $\{v,a_2\}\in E$. Then, $\{a_1,u,v\}$ is an independent set in $G$, and $u,v\in N_G(a_2)$. Thus, by Proposition \ref{prop:n3_odd}, 
     \[
        C(\lk(I_3(G),\{a_1,a_2\}))\leq d-2.
    \]

    \item[Case 3:] Assume none of the two first cases holds. Since $|V'|>\Delta$, there exists a vertex $a_2\in \lk(I_3(G),a_1)$ such that $a_2\notin N_G(a_1)$ (otherwise $\deg_G(a_1)=|N_G(a_1)|>\Delta$, a contradiction).

    Let $w\in N_G(a_2)\setminus N_G(a_1)$. Note that $w\in \lk(I_3(G),\{a_1,a_2\})$.
    
    Assume for contradiction that there exists a missing face $\tau$ of the complex $\lk(I_3(G),\{a_1,a_2\})$ that contains $w$. First, assume that $\tau=\{u,v,w\}$ is an independent set of size $3$. Then, both $u$ and $v$ must belong to $N_G(a_1)$. Otherwise, assume without loss of generality that $v\notin N_G(a_1)$. Then $\{w,v,a_1\}$ is an independent set in $G$, and therefore $\{v,w\}\notin \lk(I_3(G),\{a_1,a_2\})$, a contradiction to $\tau$ being a missing face. But then, the existence of the independent set $\{u,v,w\}$ is a contradiction to the assumption that Case $1$ does not hold.
    
    Now, assume $\tau=\{v,w\}$ is of size $2$. Then there exists an independent set $J$ of size $3$ such that $\tau\subset J\subset \tau\cup\{a_1,a_2\}$. Since $w\in N_G(a_2)$, we must have $J=\{a_1,v,w\}$. In particular $v\notin N_G(a_1)$. So, we must have $v\in N_G(a_2)$. But then, the triple $\{a_2,v,w\}$ satisfies $a_2,v,w\notin N_G(a_1)$, $\{v,w\}\notin E$, $\{a_2,v\}\in E$ and $\{a_2,w\}\in E$. This is a contradiction to the assumption that Case $2$ does not hold.
    
    Therefore, $w$ is not contained in any missing face of $\lk(I_3(G),\{a_1,a_2\})$. Let $U=N_G(a_1)\cup\{a_1,a_2\}$. Then, we have
    \[
        \lk(I_3(G),\{a_1,a_2\})=2^{N_G(a_2)\setminus N_G(a_1)}\ast \lk(I_3(G[U]),\{a_1,a_2\}).
    \]
    So, by Lemma \ref{lemma:join_with_simplex}, we have 
    \[
        C(\lk(I_3(G),\{a_1,a_2\}))
        =C(\lk(I_3(G[U]),\{a_1,a_2\})).
    \]
    By Proposition \ref{lemma:vertexset_bound}, we obtain
    \[
    C(\lk(I_3(G),\{a_1,a_2\}))\leq \frac{2|N_G(a_1)|}{3}\leq \frac{2\Delta}{3}.
    \]
    Note that $\frac{2\Delta}{3}\leq \Delta$, and $\frac{2\Delta}{3}\leq \Delta-1$ for $\Delta\geq 3$. Hence, we obtain   
    \[
    C(\lk(I_3(G),\{a_1,a_2\}))\leq \frac{2\Delta}{3}\leq d-2
    \]
    for all $\Delta\geq 2$.

\end{itemize}

For any of the three cases we have $C(\lk(I_3(G\setminus a_2),a_1))\leq d-1$ by the induction hypothesis. Also, we showed that $C(\lk(I_3(G),\{a_1,a_2\}))\leq d-2$ in all three cases. So, by Lemma \ref{lemma:tancer},
\begin{align*}
C(\lk(I_3(G),a_1))
&\leq \max\{ 
			C(\lk(I_3(G\setminus a_2),a_1)),C(\lk(I_3(G),\{a_1,a_2\}))+1
		\}
\\
&\leq d-1,
\end{align*}
as wanted.
\end{proof}

\begin{theorem}
Let $G=(V,E)$ be a graph with maximum degree $\Delta$. Then, 
\[
    C(I_3(G))\leq\begin{cases}
        \Delta+2 & \text{ if $\Delta$ is even,}\\
        \Delta+1 & \text{ if $\Delta$ is odd.}
        \end{cases}
\]
\end{theorem}
\begin{proof}
For $\Delta=1$ the claim holds by Theorem \ref{cor:bounded degree}. Assume $\Delta\geq 2$.

Let $d=\Delta+2$ if $\Delta$ is even, and $d=\Delta+1$ if $\Delta$ is odd.
We argue by induction on $|V|$. If $|V|=0$ the claim holds trivially. Otherwise, let $a_1\in V$. By the induction hypothesis, $C(I_3(G\setminus a_1))\leq d$. Also, by Proposition \ref{prop:n3_linka1}, 
$ C(\lk(I_3(G),a_1))\leq d-1$. So, by Lemma \ref{lemma:tancer},
\[
C(I_3(G))\leq \max\{ C(I_3(G\setminus a_1)), C(\lk(I_3(G),a_1))+1\}\leq d.
\]
\end{proof}

\section{Rainbow independent sets in claw-free graphs}\label{sec:rainbow}

Now we are ready to prove Theorem~\ref{maincor}.

\begin{proof}[Proof of Theorem \ref{maincor}]
We argue by induction on $n$. The case $n=1$ is trivial.
Now, assume $n>1$. 
Let $t = \left\lfloor\left(\frac{\Delta}{2}+1\right)(n-1)\right\rfloor + 1$ and let $J_1,\ldots,J_t$ be independent sets of size $n$ in $G$.
Since $t\geq \left\lfloor\left(\frac{\Delta}{2}+1\right)(n-2)\right\rfloor+1$, then, by the induction hypothesis, there exists a rainbow independent set $A$ of size $n-1$.
Without loss of generality, we may assume that $A = \{v_1,\ldots,v_{n-1}\}$, where $v_i \in J_i$ for all $i\in[n-1]$.

Let $X=\lk(I_n(G),A).$
By Proposition~\ref{prop:clawfree}, $X$ is $\left\lfloor\frac{\Delta}{2}(n-1)\right\rfloor$-collapsible. 

The family $\{J_i\}_{n\leq i\leq t}$ consists of $\left\lfloor\frac{\Delta}{2}(n-1)\right\rfloor+1$ sets not belonging to $X$. Thus, by Theorem \ref{km}, there exists a set $R=\{v_{n},\ldots,v_{t}\}$, where $v_i\in J_i$ for all $n\leq i\leq t$, such that $R\notin X$. Therefore, the set $A\cup R$ contains a set $I$ of size $n$ that is independent in $G$. $I$ is a rainbow independent set of size $n$ in $G$, as wanted.
\end{proof}

\section{Lower bounds on Leray numbers}\label{sec:lowerbound}

In this section we present some examples establishing the sharpness of our different bounds on the collapsibility of $I_n(G)$. Also, we present a family of counterexamples to Conjecture \ref{mainconj_collapsibility_version}, in the case of graphs with maximum degree at most $3$.

\subsection{Extremal examples}
Let $n$ be an integer, and $k$ be an even integer. Let $G_{k,n}$ be the graph obtained from a cycle of length $\left(\frac{k}{2}+1\right)n$ by adding all edges connecting any two vertices of distance at most $\frac{k}{2}$ in the cycle.
Note that $G_{k,n}$ is a $k$-regular graph, i.e. every vertex has degree exactly $k$.
Moreover, $G_{k,n}$ is claw-free.

In \cite{ABKK} it is shown that $f_{G_{k,n}}(n)\geq \left(\frac{k}{2}+1\right)(n-1)+1$. In particular, this shows the tightness of Theorem \ref{maincor}, in the case that $k$ is even.
Moreover, by Proposition \ref{prop:helly_corollary}, we obtain 
\[
C(I_n(G_{k,n}))\geq f_{G_{k,n}}(n)-1 \geq \left(\frac{k}{2}+1\right)(n-1).
\]
This shows that the bound in Conjecture \ref{mainconj_collapsibility_version}, whenever it holds, is tight. A different way to show this is as follows.

\begin{proposition}\label{prop:extremalex}
	\[
	\tilde{H}_i(I_n(G_{k,n}))= \begin{cases}
	\Rea & \text{ if } i=\left(\frac{k}{2}+1\right)(n-1)-1,\\
	0 & \text{ otherwise.}
	\end{cases}
	\]
	In particular, 
$
	L(I_n(G_{k,n}))\geq \left(\frac{k}{2}+1\right)(n-1).
$
\end{proposition}
\begin{proof}

Let $t=\frac{k}{2}+1$. It is easy to check that there are precisely $t$ independent sets of size $n$ in $G_{k,n}$, and they are pairswise disjoint. Therefore, $I_n(G_{k,n})$ can be described as the join of $t$ disjoint copies of the boundary of an $(n-1)$-dimensional simplex. Since the boundary of an $(n-1)$-dimensional simplex is an $(n-2)$-dimensional sphere, we obtain
by Theorem \ref{thm:kunneth}:
	\[
	\tilde{H}_i(I_n(G_{k,n}))= \begin{cases}
	\Rea & \text{ if } i=t(n-1)-1,\\
	0 & \text{ otherwise.}
	\end{cases}
	\]
Thus, $L(I_n(G)) \geq t(n-1)= \left(\frac{k}{2}+1\right)(n-1)$.

\end{proof}
	By Lemma \ref{lem:coll_leray}, we obtain
	\[
			C(I_n(G_{k,n}))\geq L(I_n(G_{k,n}))\geq \left(\frac{k}{2}+1\right)(n-1).
	\]
	On the other hand, $I_n(G_{k,n})$ is a $\left(\left(\frac{k}{2}+1\right)(n-1)-1\right)$-dimensional complex, and therefore it is $\left(\frac{k}{2}+1\right)(n-1)$-collapsible. So,
	\[
		C(I_n(G_{k,n}))=\left(\frac{k}{2}+1\right)(n-1).
	\]

Proposition~\ref{prop:extremalex} also shows that the bound in Proposition~\ref{prop:easy} is tight, since $G_{2k-2,n}$ is a $k$-partite graph with $C(I_n(G_{2k-2,n}))=k(n-1)$. 
Another such extremal example is the complete $k$-partite graph $K_{n,\ldots,n}$.
In this case, it easy to see that $I_n(K_{n,\ldots,n}) \cong I_n(G_{2k-2,n})$.

\subsection{A counterexample to Conjecture \ref{mainconj_collapsibility_version}}
Let $G=(V,E)$ be the dodecahedral graph. It will be convenient to represent $G$ as a generalized Petersen graph (see \cite{watkins1969theorem}), as follows:
\[
	V=\{a_1,\ldots,a_{10},b_1,\ldots,b_{10}\}
\]
and
\[
	E=\{\{a_i,b_i\}, \,\{a_i,a_{i+1}\},\, \{b_i,b_{i+2}\} :\,  i=1,2,\ldots,10\},
\]
where the indices are taken  modulo $10$.

Every vertex in $G$ is adjacent to exactly $3$ vertices; that is, $G$ is $3$-regular. The maximal independent sets in $G$ are the sets
	\[
	I_i= \{a_i,a_{i+2},a_{i+5},a_{i+7},b_{i-2},b_{i-1},b_{i+3},b_{i+4}\}
	\]
for $i=1,\ldots,5$ (also here, the indices are to be taken modulo $10$). In particular, $\alpha(G)=8$.

\begin{proposition}\label{claim:10,2}
	Let $G=(V,E)$ be the dodecahedral graph. Then,
	\[
	\tilde{H}_i(I_8(G))= \begin{cases}
	\Rea^4 & \text{ if } i=15,\\
	0 & \text{ otherwise.}
	\end{cases}
	\]
	In particular, $L(I_8(G))\geq 16$.
\end{proposition}
\begin{proof}

Let $\mathcal{F}=\{V\setminus I_1, V\setminus I_2,\ldots, V\setminus I_5\}$.  The family $\mathcal{F}$ is the set of maximal faces of $D(I_8(G))$. So, by the Nerve Theorem (Theorem \ref{thm:nerve}),
\[
	\tilde{H}_i(N(\mathcal{F})) \cong \tilde{H}_i(D(I_8(G)))
\]
for all $i\geq -1$. So, by Alexander duality (Theorem \ref{thm:alexander}), 
\begin{equation}\label{eq:dual}
	\tilde{H}_i(N(\mathcal{F}))=\tilde{H}_{|V|-i-3}(I_8(G))=\tilde{H}_{17-i}(I_8(G))
\end{equation}
for all $-1\leq i\leq |V|-2 = 18$.
We have
\[
	N(\mathcal{F})= \left\{ A\subset [5]: \,  \bigcap_{i\in A} V\setminus I_i \neq \emptyset\right\} = \left\{ A\subset[5] : \, \bigcup_{i\in A} I_i \neq V\right\}.
\]
It is easy to check that $N(\mathcal{F})$ is the complete $2$-dimensional complex on $5$ vertices. So,
\[
    \tilde{H}_i(N(\mathcal{F}))=\begin{cases}
    \Rea^4 & \text{ if } i=2,\\
    0 & \text{ otherwise}.
    \end{cases}
\]
Thus, by \eqref{eq:dual},
\[
\tilde{H}_i(I_8(G))= \begin{cases}
\Rea^4 & \text{ if } i=15,\\
0 & \text{ otherwise,}
\end{cases}
\]
as wanted.
\end{proof}

We obtain $C(I_8(G))\geq L(I_8(G))\geq 16> 2\cdot (8-1)= 14$. Therefore, $I_8(G)$ does not satisfy the bound in Conjecture \ref{mainconj_collapsibility_version}.
However, this is not a counterexample for Conjecture~\ref{mainconj}.
Indeed, it is not hard to check that $f_G(8) = 11$.

\subsection{Leray number of the disjoint union of graphs}

The following result will help us in constructing more counterexamples to Conjecture \ref{mainconj_collapsibility_version}:

\begin{theorem}
\label{prop:union}
Let $G$ be the disjoint union of the graphs $G_1,\ldots,G_m$.
For $i\in[m]$, let $t_i= \alpha(G_i)$ and let $\ell_i=L(I_{t_i}(G_i))$. Let $t=\sum_{i=1}^m t_i = \alpha(G)$ and $\ell=L(I_t(G))$.
Then,
\[
    \ell= \sum_{i=1}^m \ell_i + m -1.
\]
\end{theorem}

The proof relies on the following result.

\begin{proposition}\label{prop:union_homology}
Let $G$ be the disjoint union of the graphs $G_1,\ldots,G_m$.
For $i\in[m]$, let $t_i= \alpha(G_i)$. Let $t=\sum_{i=1}^m t_i = \alpha(G)$. Then,
$
\tilde{H}_k(I_t(G))=0
$
if and only if for every choice of integers $k_1,\ldots, k_m$ satisfying $\sum_{i=1}^m k_i = k-2m+2$, 
$
\tilde{H}_{k_i}(I_{t_i}(G_i))=0
$
for all $i\in[m]$.
	
\end{proposition}
\begin{proof}
For all $i\in[m]$, let $V_i$ be the vertex set of $G_i$, and let $V=\bigcup_{i=1}^m V_i$ be the vertex set of $G$. Let $N_i=|V_i|$ for all $i\in[m]$, and $N=|V|=\sum_{i=1}^m N_i$.

A set $U\subset V$ contains an independent set of size $t$ in $G$ if and only if $U\cap V_i$ contains an independent set of size $t_i$ in $G_i$ for all $i\in[m]$. That is, $U\notin I_t(G)$ if and only if $U\cap V_i\notin I_{t_i}(G_i)$ for all $i\in[m]$.
Equivalently, a set $W\subset V$ belongs to $D(I_t(G))$ if and only if $W\cap V_i\in D(I_{t_i}(G_i))$ for all $i\in[m]$. Thus, we have
\[
	D(I_t(G))= D(I_{t_1}(G_1))\ast \cdots \ast D(I_{t_m}(G_m)).
\]
Note that for every $i\in [m]$, $V_i\notin I_{t_i}(G_i)$ (since $G_i$ contains an independent set of size $t_i=\alpha(G_i)$). Similarly, $V\notin I_t(G)$. So, by Alexander duality (Theorem \ref{thm:alexander}), we have
\[
\tilde{H}_{j}(D(I_{t_i}(G_i)))= \tilde{H}_{N_i-j-3}(I_{t_i}(G_i))
\]
for all $i\in[m]$ and $-1\leq j\leq |V_i|-2$, and
\[
\tilde{H}_{j}(D(I_{t}(G)))= \tilde{H}_{N-j-3}(I_{t}(G))
\]
for all $-1\leq j\leq |V|-2$.

Therefore, by Theorem \ref{thm:kunneth}, we obtain
\begin{align*}
&\tilde{H}_{N-j-3}(I_{t}(G))
=\tilde{H}_{j}(D(I_{t}(G)))
\\
&=
\bigoplus_{j_1+\cdots+j_m=j-m+1} \tilde{H}_{j_1}(D(I_{t_1}(G_1)))\otimes \cdots \otimes \tilde{H}_{j_m}(D(I_{t_m}(G_m)))
\\
&=
\bigoplus_{j_1+\cdots+j_m=j-m+1} \tilde{H}_{N_1-j_1-3}(I_{t_1}(G_1))\otimes \cdots \otimes \tilde{H}_{N_m-j_m-3}(I_{t_m}(G_m)).
\end{align*}
Setting $k=N-j-3$ and $k_i=N_i-j_i-3$ for all $i\in[m]$, we obtain
\[
\tilde{H}_{k}(I_{t}(G))
=
\bigoplus_{k_1+\cdots+k_m=k-2m+2} \tilde{H}_{k_1}(I_{t_1}(G_1))\otimes \cdots \otimes \tilde{H}_{k_m}(I_{t_m}(G_m)).
\]
In particular, $
\tilde{H}_k(I_t(G))=0
$
if and only if for every choice of $k_1,\ldots, k_m$ satisfying $\sum_{i=1}^m k_i = k-2m+2$, 
$
\tilde{H}_{k_i}(I_{t_i}(G_i))=0
$
for all $i\in[m]$.

\end{proof}

\begin{proof}[Proof of Theorem \ref{prop:union}]
	
For all $i\in[m]$, let $V_i$ be the vertex set of $G_i$, and let $V=\bigcup_{i=1}^m V_i$ be the vertex set of $G$.
	
Since $L(I_t(G))=\ell$, there exists a subset $U\subset V$ such that
\[
	\tilde{H}_{\ell-1}(I_t(G[U]))\neq 0.	
\]
Let $G'=G[U]$ and $G'_i=G_i[U\cap V_i]$ for all $i\in[m]$. Note that $I_t(G')$ is not the complete complex, since it has non-trivial homology; hence, $\alpha(G')=t$. Since $G'$ is the disjoint union of the graphs $G'_1,\ldots, G'_m$, we must have $\alpha(G'_i)=t_i$ for all $i\in[m]$. By Proposition \ref{prop:union_homology}, there exists $k_1,\ldots, k_m$ satisfying $\sum_{i=1}^m k_i = \ell-2m+1$ such that
\[
	\tilde{H}_{k_i}(I_{t_i}(G'_i))\neq 0.
\]
In particular, $\ell_i=L(I_{t_i}(G_i))\geq k_i+1$. Summing over all $i\in[m]$, we obtain
\[
	\sum_{i=1}^m \ell_i \geq \sum_{i=1}^m k_i + m = \ell-m+1.
\]

Now, let $i\in [m]$. Since $\ell_i=L(I_{t_i}(G_i))$, there exists a subset $U_i\subset V_i$ such that
\[
	\tilde{H}_{\ell_i-1}(I_{t_i}(G_i[U_i]))\neq 0.
\]
Let $G'_i=G_i[U_i]$. Note that $I_{t_i}(G'_i)$ is not the complete complex, since it has non-trivial homology. Therefore, $\alpha(G'_i)=t_i$.
Let $U=U_1\cup\cdots\cup U_m$, and let $G'=G[U]$. Then, $G'$ is the disjoint union of $G'_1,\ldots,G'_m$. By Proposition \ref{prop:union_homology}, we have
\[
	\tilde{H}_{\sum_{i=1}^m (\ell_i-1)+2m-2}(I_t(G'))=\tilde{H}_{\sum_{i=1}^m \ell_i+m-2}(I_t(G'))\neq 0.
\]
Thus, $\ell=L(I_t(G))\geq \sum_{i=1}^m \ell_i +m -1$.

\end{proof}

\begin{corollary}\label{cor:dodecahedral}
	Let $G_k$ be the union of $k$ disjoint copies of the dodecahedral graph. Then,
\[
L(I_{8k}(G_k))\geq 17k-1.
\]
\end{corollary}
\begin{proof}
Let $H_1,\ldots, H_k$ be $k$ disjoint copies of the dodecahedral graph. Then, by Propositions \ref{claim:10,2} and \ref{prop:union}, we obtain
\begin{multline*}
	L(I_{8k}(G_k))=L(I_{8k}(H_1\cup H_2\cup \cdots \cup H_k))
	\\= \sum_{i=1}^k L(I_8(H_i))+k-1 \geq 16k+k-1 =17k-1.
\end{multline*}
\end{proof}

Note that the graphs $G_k$ are $3$-regular, and
\[
\frac{L(I_{8k}(G_k))}{8k-1} \geq \frac{17k-1}{8k-1}>  2\frac{1}{8}> 2.
\]
Thus, the complexes $I_{8k}(G_k)$ do not satisfy the bound in Conjecture \ref{mainconj_collapsibility_version}.

Note that the graphs $G_k$ are \emph{not} counterexamples for Conjecture~\ref{mainconj}.
This can be shown by the following observation.
\begin{proposition}\label{prop:disjoint_union_f}
	Let $G$ be the disjoint union of two graphs $G_1$ and $G_2$ with $\alpha(G_1)=t_1$ and $\alpha(G_2)=t_2$. Then,
	\[
		f_G(t_1+t_2)\leq \max \{ f_{G_1}(t_1), f_{G_2}(t_2)+t_1\}.
	\]
\end{proposition}
\begin{proof}
Let $V_1$ and $V_2$ denote the vertex sets of $G_1$ and $G_2$ respectively.
Let $t= \max \{ f_{G_1}(t_1), f_{G_2}(t_2)+t_1\}$.

Let $\mathcal{A}=\{A_1,\ldots, A_t\}$ be a family of independent sets of size $t_1+t_2$ in $G$.
Note that any independent set of size $t_1+t_2=\alpha(G)$ in $G$ has $t_1$ vertices in $V_1$ and $t_2$ vertices in $V_2$.

Thus, $A_1\cap V_1, A_2\cap V_1,\ldots, A_t\cap V_1$ is a family of  $t\geq f_{G_1}(t_1)$ independent sets of size $t_1$ in $G_1$. Hence, it contains a rainbow independent set $R_1$ of size $t_1$. Without loss of generality, we may assume that $R_1= \{a_{t-t_1+1},\ldots,a_{t}\}$, where $a_i \in A_i$ for all $i\in\{t-t_1+1,\ldots,t\}$.

The family $A_1\cap V_2, A_2\cap V_2,\ldots,A_{t-t_1}\cap V_2$ is a family of  $t-t_1\geq f_{G_2}(t_2)$ independent sets of size $t_2$ in $G_2$; therefore, it contains a rainbow independent set $R_2$ of size $t_2$.

Then, the set $R_1\cup R_2$ is a rainbow independent set of size $t_1+t_2$ in $G$ with respect to $\mathcal{A}$, as wanted.
\end{proof}

Applying Proposition \ref{prop:disjoint_union_f} repeatedly, we obtain that $f_{G_k}(8k) \leq 8k+3 < 16k-1$.

\section{Open problems}\label{sec:open}

We showed that the bound in Conjecture \ref{mainconj_collapsibility_version} holds in some special cases, but not in general. It would be interesting to decide for which values of $\Delta$ and $n$ the inequality holds. Alternatively, one could try to characterize the graphs satisfying the bound for all values of $n$.

A weaker result, which may hold for general bounded degree graphs, is the following:
\begin{conjecture}\label{conjecture:link}
Let $G=(V,E)$ be a graph with maximum degree at most $\Delta$, and let $n\geq 1$ be an integer. Let $A$ be an independent set of size $n-1$ in $G$. Then,
\[
    C(\lk(I_n(G),A))\leq \left\lfloor\frac{(n-1)\Delta}{2}\right\rfloor.
\]
\end{conjecture}
For the subclass of claw-free graphs, this is proved in Proposition \ref{prop:clawfree}. Conjecture \ref{conjecture:link} would imply the bound $f_G(n)\leq\left\lfloor\left(\frac{\Delta}{2}+1\right)(n-1)\right\rfloor+1$ (by the same argument as the one used to prove Theorem \ref{maincor}), settling Conjecture \ref{mainconj} in the case of even $\Delta$.

Another possible direction is to focus on the family of claw-free bounded degree graphs. We showed in Theorem \ref{maincor} that Conjecture \ref{mainconj} holds for graphs in this family when $\Delta$ is even. In the case of odd $\Delta$, although we obtain good upper bounds for $f_G(n)$, the question remains unsettled. It would also be interesting to prove the corresponding tight upper bound on the collapsibility number of $I_n(G)$, at least for the case of even $\Delta$.

We know, by Proposition \ref{claim:10,2}, that Conjecture \ref{mainconj_collapsibility_version} does not hold for graphs with maximum degree at most $3$. The following problem arises:
\begin{problem}
	Find the smallest positive integer $g(n)$ such that the following holds: for every graph $G$ with maximum degree at most $3$, \[C(I_n(G)) \leq g(n).\]
\end{problem}
By Theorem \ref{cor:bounded degree} and Proposition \ref{prop:extremalex} we have $2(n-1)\leq g(n)\leq 3(n-1)$ for all $n\geq 1$, and, by Corollary \ref{cor:dodecahedral}, $g(8k)\geq 17k-1$ for all $k\geq 1$. Improving either the upper or lower bounds for $g(n)$ may be of interest.

\section*{Acknowledgment}
The authors thank Jinha Kim for her insightful comments during the early stages of this research.

\bibliographystyle{abbrv}
\bibliography{biblio}

\begin{thebibliography}{10}

\bibitem{adamaszek2014extremal}
M.~Adamaszek.
\newblock Extremal problems related to betti numbers of flag complexes.
\newblock {\em Discrete Appl. Math.}, 173:8--15, 2014.

\bibitem{aharoni2009rainbow}
R.~Aharoni and E.~Berger.
\newblock Rainbow matchings in $ r $-partite $ r $-graphs.
\newblock {\em Electron. J. Combin.}, 16(1):R119, 2009.

\bibitem{aharoni2016large}
R.~Aharoni, E.~Berger, M.~Chudnovsky, D.~Howard, and P.~Seymour.
\newblock Large rainbow matchings in general graphs.
\newblock {\em European J. Combin.}, 79:222 -- 227, 2019.

\bibitem{ABKK}
R.~Aharoni, J.~Briggs, J.~Kim, and M.~Kim.
\newblock Rainbow independent sets in certain classes of graphs.
\newblock preprint, \url{https://arxiv.org/abs/1909.13143}, 2019.

\bibitem{aharoni2018fractional}
R.~Aharoni, R.~Holzman, and Z.~Jiang.
\newblock Rainbow fractional matchings.
\newblock {\em Combinatorica}, 2019.

\bibitem{barat2017rainbow}
J.~Bar{\'a}t, A.~Gy{\'a}rf{\'a}s, and G.~N. S{\'a}rk{\"o}zy.
\newblock Rainbow matchings in bipartite multigraphs.
\newblock {\em Period. Math. Hungar.}, 74(1):108--111, 2017.

\bibitem{bjorner1995topological}
A.~Bj{\"o}rner.
\newblock Topological methods.
\newblock In R.~Graham, M.~Gr{\"o}tschel, and L.~Lov{\'a}sz, editors, {\em
  Handbook of combinatorics}, chapter~34, pages 1819--1872. North-Holland,
  Amsterdam, 1994.

\bibitem{bjorner2009note}
A.~Bj{\"o}rner and M.~Tancer.
\newblock Note: Combinatorial {A}lexander duality-- a short and elementary
  proof.
\newblock {\em Discrete Comput. Geom.}, 42(4):586, 2009.

\bibitem{briggs2019choice}
J.~Briggs and M.~Kim.
\newblock Choice functions in the intersection of matroids.
\newblock {\em Electron. J. Combin.}, 26(4):P4.26, 2019.

\bibitem{brooks1941colouring}
R.~L. Brooks.
\newblock On colouring the nodes of a network.
\newblock {\em Proc. Cambridge Philos. Soc.}, 37:194--197, 1941.

\bibitem{KM2005}
G.~Kalai and R.~Meshulam.
\newblock A topological colorful {H}elly theorem.
\newblock {\em Adv. Math.}, 191(2):305--311, 2005.

\bibitem{khmel}
I.~Khmelnitsky.
\newblock $d$-collapsibility and its applications.
\newblock Master's thesis, Technion, Haifa, 2018.

\bibitem{lekkeikerker1962representation}
C.~G. Lekkerkerker and J.~C. Boland.
\newblock Representation of a finite graph by a set of intervals on the real
  line.
\newblock {\em Fund. Math.}, 51(1):45--64, 1962.

\bibitem{tancer2011strong}
M.~Tancer.
\newblock Strong $d$-collapsibility.
\newblock {\em Contrib. Discrete Math.}, 6(2):32--35, 2011.

\bibitem{watkins1969theorem}
M.~E. Watkins.
\newblock A theorem on {T}ait colorings with an application to the generalized
  {P}etersen graphs.
\newblock {\em J. Combin. Theory}, 6(2):152--164, 1969.

\bibitem{Weg75}
G.~Wegner.
\newblock $d$-collapsing and nerves of families of convex sets.
\newblock {\em Arch. Math.}, 26(1):317--321, 1975.

\end{thebibliography}

\end{document}